\documentclass[11pt]{article}
\usepackage{amssymb}
\usepackage{amsmath}
\usepackage{amsfonts}
\usepackage{graphicx}
\usepackage{color}
\usepackage[round]{natbib}
\usepackage{latexsym}
\usepackage{epstopdf}
\usepackage[all]{xy}

\newtheorem{thm} {\noindent \bf{Theorem}}

\newtheorem{lem}{\noindent \bf{Lemma}}

\newenvironment{proof}[1][Proof]{\noindent \textbf{#1.} }{\ \rule{0.5em}{0.5em}}
\newcounter{remno} 
\begin{document}
\title{Inventory Policies for Two Products under Poisson Demand: Interaction between Demand Substitution, Limited Storage Capacity and Replenishment Time Uncertainty}

\author{Apostolos Burnetas \\
Department of Mathematics, University of Athens\\
Athens, Greece\\
aburnetas@math.uoa.gr\\
\and
Odysseas Kanavetas\\
Faculty of Engineering and Natural Sciences, Sabanci University\\
Istanbul, Turkey\\
okanavetas@sabanciuniv.edu}
\medskip

\maketitle

\begin{abstract}
We consider a two-product inventory system with independent Poisson demands, limited joint storage capacity and partial demand substitution. Replenishment is performed simultaneously for both products and the replenishment time may be fixed or exponentially distributed. For both cases we develop a Continuous Time Markov Chain model for the inventory levels and derive expressions for the expected profit per unit time. We prove that the profit function is  submodular in the order quantities, which allows for a more efficient algorithm to determine the optimal ordering policy. Using computational experiments we assess the effect of substitution and replenishment time uncertainty on the order quantities and the profit as a function of the storage capacity.

\end{abstract}

\section{Introduction}\label{intro}
In many multi-product inventory management systems several factors must be considered when designing ordering policies. 
First and foremost, when the demand is stochastic, the system is subject to shortages in any time interval between two successive replenishments, whereas  high stocking levels may incur unnecessary inventory costs. The problem is intensified when there is also uncertainty in the replenishment times. When products are stored and sold from a common facility with limited capacity they compete for storage space, which in turn introduces dependencies in the inventory process even though the individual demand processes may be independent. On the other hand, there is often some degree of demand substitution between different products.  The substitution possibility  must generally be taken into account when designing effective inventory replenishment policies, as it could attenuate the constraints imposed by limited storage capacity.  

In this paper we study the interactions between demand and replenishment time uncertainty, constrained storage capacity and demand substitution, in a two-product joint replenishment system with lost sales and base-stock replenishment. The base-stock  framework is appropriate when there are small or no fixed ordering costs. This is a reasonable assumption when replenishment epochs are exogenously planned and not determined by inventory levels, for example in perishable product settings, or when a supplier makes periodic deliveries according to her own schedule and retailers can purchase product to replenish their stock only at those times. We assume that the two products are purchased at fixed unit costs and sold at fixed retail prices, thus the shortage cost is manifested as foregone profit due to lost sales. On the other hand, to model inventory costs, we impose a, generally product dependent, inventory holding cost for each unsold unit of product at the end of a replenishment cycle. 

The essential feature of our model is the two-way partial demand substitution. We assume that when a customer does not find the product he prefers in stock while the other product is available, he will buy it with a given probability. By indirectly introducing a degree of demand pooling between different products, substitution helps deal with demand uncertainty, especially in situations where storage capacity is tight and wrong stocking decisions  may lead to severe shortages in some products and at the same time unsold quantities in others. On the other hand, accounting for substitution in designing inventory policies introduces substantial complications in the analysis. The main reason is that in general whether a customer will buy a substitute product depends on whether his original choice is available at the time of his arrival to the system. Therefore, in general, sales, shortage and unsold quantities depend not only on the total demand of each product during a replenishment cycle, but also on the timing of arrivals of individual demands of each product. For this reason, a usual approximation in analyzing inventory problems under substitution is to assume that the quantities demanded as substitutes are a fixed deterministic proportion of the demand of one product which is added to the total demand of the other product. 

The main contribution of our model is that it endogenizes the substitution process by modeling the demand arrivals for the two products as independent Poisson processes. This allows for a more realistic assessment of the costs of a replenishment policy and, thus, of the resulting optimal policies. On the other hand, analytic expressions for the profit function are no longer feasible. We instead prove structural properties which in turn enable a more efficient search for optimal policies. 

To model capacity or other restrictions on the replenishment policy, we impose a linear constraint on the order quantities. Finally, we consider two cases for the replenishment epochs. First we assume that the replenishment cycles have fixed constant length. We then consider the case where replenishment epochs occur according to a Poisson process, independent of the inventory levels. For both cases we develop a Markov Process model for the joint inventory levels, derive pertinent expressions for the profit function, and prove that it is submodular in the order quantities. The submodularity property implies a monotonicity structure for the optimal policy, which in turn allows for more efficient computations. 

In addition to the structural results, we conduct several numerical experiments in order to gain insights on the interactions between the various aspects of the model, and mainly between the limitations in  storage capacity, the substitution effects and the uncertainty, or lack thereof, in the replenishment process. Some of the insights obtained are the following. The benefit of demand substitutability is not significant when the available capacity is very low, and thus substantial  shortages cannot be avoided, as well as when the capacity is very high, and thus it does not impose severe restrictions on the order quantities.  However it may  become substantial in intermediate capacity levels, depending on the relative cost parameters of the two products. In terms of the effect on the optimal ordering policy, this also depends on the similarity of the products in terms of their profit/cost parameters. In extreme cases where the two products are significantly different, the existence of substitution effects may result in complete abandonment of the less favorable product. 
Finally, regarding the uncertainty in the replenishment cycle, its effect on the profit is generally detrimental. The effect on the optimal policy also depends on the economic parameters of the two products. Specifically, when the shortage costs are significant compared to the holding costs, the effect of the replenishment time uncertainty is to increase the order quantities, in order to mitigate the increased risk of stockouts, while the opposite effect appears when the holding costs are more significant. 

Many  models of ordering and inventory management with various levels of substitution have been proposed and analyzed in the literature. In a significant number of them it is approximately assumed that the substituted demand of a product is a fixed fraction of the unsatisfied demand of that product. Under no salvage value and lost sale penalties, \cite{parlar} develop an expression of  the profit function, show that it is concave under appropriate conditions and derive optimality conditions for the ordering policy. Assuming positive salvage values and shortage costs, \cite{Khouja} also derive the profit function and establish upper and lower bounds for the optimal ordering quantities.  \cite{stavrulaki} and \cite{krommyda} consider extensions of the ordering problem of substitutable products for stock-dependent demand. In \cite{stavrulaki} the joint effect of demand stimulation and product substitution on inventory decisions by considering a single-period, stochastic demand setting. In \cite{krommyda} the demand of each product is a deterministic function of both inventory levels and two-way substitution is assumed. \cite{netessine2003} also assume deterministic substitution proportions and in addition to the centralized model they also analyze a competitive situation where the substitutable products are stored by different retailers.  

Several approaches have also been developed to model the substitution process in more detail. One such approach employs two-stage stochastic programming, where production/stocking decisions are made in the first stage before demand is realized, where as in the second stage demand allocation to direct and substitute sales is performed as a recourse action. \cite{rao2004multi} consider a single period, multi-product model with fixed production costs and one-way substitution. They develop a mixed integer two-stage stochastic programming model, exploit structural properties and develop several heuristics. \cite{vaagen2011modelling} also consider a multi-product setting with partial substitution possible between any two products. They propose linear and mixed-integer stochastic programming formulations and obtain several managerial insights from solving test problems. 

A different line of works more related to ours addresses the product choice of arriving customer in a dynamic fashion. In \cite{mahajan2001stocking} a newsvendor setting is adopted, where heterogeneous customers dynamically substitute among available products, based on maximizing a utility function. A sample path gradient algorithm is applied to compute inventory levels and Poisson customer arrivals are simplified by normal approximations. \cite{dong2009dynamic} also allow dynamic customer choice, but they focus on the problem of maximizing profit through dynamic pricing in a newsvendor framework. They develop a dynamic programming model for the pricing problem and propose heuristic methods for the initial inventory decisions. \cite{shumsky2009dynamic} and \cite{xu2011optimal} consider dynamic strategies of offering substitution to arriving customers. In \cite{shumsky2009dynamic} demand classes are downward substitutable and customers can be upgraded by at most one level. The optimal substitution policy is shown to have a protection limit structure. \cite{xu2011optimal} consider a nonstationary Poisson process and random batch sizes. They develop a stochastic dynamic programming model and prove threshold properties for the optimal substitution policy. 

Recently,  \cite{salameh} study the two-product joint replenishment model with substitution in an economic order quantity framework. The EOQ assumptions allow to explicitly model the stock evolution prior to and after the depletion of one of the products. It is shown that including the substitution effect in the ordering problem may result in substantial improvements in the total cost compared to the ordinary joint replenishment model. Finally, \cite{deflem2011} develop a discrete-time Markov Chain model for a two-product periodic review system with centrally controlled one way downward substitution. The cost benefit of the substitution policy compared to no-pooling or full-pooling approaches is assessed numerically and it is shown that substitution can outperform both.  
Our paper contributes to the literature in three ways. First we develop Markov Chain models for a two-product system with partial two-way substitution, under fixed and random replenishment times, and develop analytic expressions for the cost as a function of the order quantities. Second we establish the submodularity of the cost function which allows for a more efficient computation of the optimal ordering policy, and third, we obtain several managerial insights on the interaction of substitution with the inventory capacity constraint. 

The paper is structured as follows . 
In Section \ref{model} we present the two approaches according to the replenishment types. In Section 3 we derive the transient distribution of the number of remaining products after a deterministic replenishment time and we prove the submodularity of the profit function. In Section 4 we derive the stationary distribution of the number of remaining products when the replenishment time is exponentially distributed and we show that the profit function is also submodular. In Section 5 we describe an algorithm for the computation of the optimal ordering quantities. In Section 6 we present computational results and managerial insights and in Section 7 conclusions and extensions.

\section{Model description}\label{model}

We consider a retailer who orders and sells two partially substitutable products. Orders for both products are placed at replenishment epochs. The time between two successive replenishments is referred to as a period. Let $r_i, c_i, h_i$ denote the unit retail price, wholesale price and inventory holding cost for quantities in inventory at the end of each period, respectively, for product $i, i=1,2$. Order sizes are denoted by $Q_1, Q_2$. We assume that due to limited storage capacity or other constraints the order quantities are restricted by a linear inequality  $a_1 Q_1 + a_2 Q_2\leq C$, where $a_1, a_2$ nonnegative constants. For example, for $a_1=a_2=1$, the inequality may reflect a storage capacity constraint, and for $a_i=c_i, i=1,2$, a purchasing budget constraint.

Between the two products there is two-way partial substitution. Specifically, we assume that a customer who originally intends to buy product $i$ will switch to product $j$ with probability $p_{ij}, i\neq j$, if the original product is not available. The substitution probabilities $p_{12}, p_{21}$ are exogenous parameters.

We next consider the demand process. Because of the substitution possibility, the actual sales, unsold quantities and shortages of both products, and as a result the profits/costs depend not only on the demand of each, but also on the timing of individual customer arrivals. In fact this is one of the complicating factors in evaluating and optimizing ordering policies, and as such it has been approximated by various assumptions in the literature.
In this paper we model the demand process explicitly, by assuming that the customers who originally intend to buy product $i$, arrive according to a Poisson process with 
 rate $\lambda_i$, $i=1,2$. When a demand of one unit of product $i$ comes and there is no product $i$ available, the customer accepts to buy one unit of product $j$, $j\neq i$, with probability $p_{ij}$, $i,j=1,2$.
The advantage of this modeling assumption is that the effect of the timing of arrivals on the substitution process is endogenized and thus represented more accurately.

Regarding the replenishment period, we consider two approaches, which lead to separate stochastic models.
In the first approach we assume that the duration of the period is deterministic and equal to $T$. This corresponds to the more common situation where the replenishment occurs at fixed regular intervals, such as on  a daily or weekly basis. In the second approach we consider the case where the replenishment time is exponentially distributed with parameter $\mu$, independent of the demand process. The exponential distribution assumption on the one hand allows us to model the inventory process in a computationally tractable  Markovian setting. On the other hand it captures situations where the replenishment process is not completely reliable and can occasionally be very late.

Our aim is to determine the quantities $Q_1$ and $Q_2$ that maximize the expected profit per unit time. To this end, we develop a stochastic model of the inventory process for each of the two approaches about the replenishment time, derive an analytic expression for the profit function for each model and show that it is submodular in the order quantities. The submodularity property leads to a faster algorithm for computing the optimal values of $Q_1$ and $Q_2$.

In the following we will refer to the interval between two successive replenishment epochs as a cycle or period. Since both products are replenishable, the replenishment instants constitute regeneration epochs for the inventory process. Therefore, the profits during successive cycles are independent and identically distributed, thus the expected profit per unit time can be expressed as the ratio of expected cycle profit over expected cycle length.

Specifically, let $(N_1, N_2)$ denote the inventory levels of the two products at the end of a cycle, right before the next replenishment occurs. Then the quantities sold during the cycle are equal to $Q_i-N_i, i=1,2$ and  the cycle profit can be expressed as
\begin{eqnarray*}
\Pi &=& r_1 (Q_1-N_1) + r_2 (Q_2-N_2) - c_1 Q_1 - c_2 Q_2 - h_1 N_1 - h_2 N_2\\
&=&(r_1 -c_1) Q_1 + (r_2-c_2) Q_2 - (r_1+h_1) N_1 - (r_2+h_2) N_2.
\end{eqnarray*}

This expression is valid under both assumptions on cycle duration. However the probability distribution of $(N_1, N_2)$ is different between the two approaches, and as a result the expected profit per period has a different analytic expression. In order to distinguish the two cases and avoid confusion in the subsequent analysis, we will denote the inventory levels under fixed replenishment time as $(n_1, n_2)$, and under random replenishment times as $(m_1, m_2)$. Using this notation, the expected profit per period under fixed replenishment time is equal to
\begin{eqnarray}
\pi^{(1)}(Q_1, Q_2) &=& \frac{1}{T} [
(r_1 -c_1) Q_1 + (r_2-c_2) Q_2  \nonumber \\
&&- (r_1+h_1) E_{Q_1, Q_2} (n_1) - (r_2+h_2) E_{Q_1, Q_2} (n_2) ] ,
\label{profdef1}
\end{eqnarray}
where $E_{Q_1, Q_2}$ denotes expectation under replenishment policy $(Q_1, Q_2)$.
For the case of random replenishment time, the expected cycle length is equal to $\mu^{-1}$, therefore the corresponding expression for the profit rate is
\begin{eqnarray}
\pi^{(2)}(Q_1, Q_2) &=& \mu [
(r_1 -c_1) Q_1 + (r_2-c_2) Q_2  \nonumber \\
&&- (r_1+h_1) E_{Q_1, Q_2} (m_1) - (r_2+h_2) E_{Q_1, Q_2} (m_2) ] .
\label{profdef2}
\end{eqnarray}

Under both approaches, it is desired to identify the optimal stocking levels $Q_1, Q_2$ that maximize the expected profit per unit time, subject to the stocking constraint, i.e.,
\begin{equation}
  \label{eq:optim_problem}
  \max\{ \pi^{(i)}(Q_1, Q_2): a_1 Q_1 + a_2 Q_2 \leq C \},
\end{equation}
for $i=1, 2$.

In the remainder of this section we develop the stochastic model that describes the inventory process and derive analytic expressions for the profit functions.

First consider the case of fixed replenishment time $T$. Let $(n_1(t), n_2(t))$ denote the inventory levels at time $t$. The stochastic process $\{(n_1(t), n_2(t)), \ t\geq 0\}$ is a continuous-time Markov process, however it is not stationary because the time parameter $t$ affects the time since the last replenishment. On the other hand, since the process regenerates at replenishment epochs and we are interested in the probabilistic behavior during a single cycle, it is sufficient to consider the probabilistic behavior of the system during the interval $[0,T]$ assuming initial state $(Q_1, Q_2)$. During this interval the inventory evolves according to a continuous-time Markov chain with state space $S=\{ (i_1,i_2):i_1 = 0,1,\ldots,Q_1,\ i_2=0,1,\ldots,Q_2 \}$ and the transition rate diagram presented in Figure \ref{fixedTrates}. Note that state $(0,0)$ is considered an absorbing state in this model, since if the inventory is depleted during a cycle, the next state transition will occur after the end of the cycle.

\begin{figure}%
\begin{center}
\tiny{\xymatrix{*+<2pt>[F-:<20pt>]{Q_1,Q_2}\ar@/^1pc/[r]^{\lambda_{2}}\ar@/^1pc/[d]^>>>>>>{\lambda_1}&
*+<2pt>[F-:<20pt>]{Q_1,Q_2-1}\ar@/^1pc/[r]^{\lambda_{2}}\ar@/^1pc/[d]^>>>>>>{\lambda_{1}}&
*+<2pt>[F-:<20pt>]{Q_1,Q_2-2}\ar@/^1pc/[d]^>>>>>>{\lambda_{1}}&...&
*+<2pt>[F-:<20pt>]{Q_1,1}\ar@/^1pc/[r]^{\lambda_{2}}\ar@/^1pc/[d]^>>>>>>{\lambda_{1}}&
*+<2pt>[F-:<20pt>]{Q_1,0}\ar@/^1pc/[d]^>>>>>>{s_1}\\
*+<2pt>[F-:<20pt>]{Q_1-1,Q_2}\ar@/^1pc/[r]^{\lambda_{2}}\ar@/^1pc/[d]^>>>>>>{\lambda_1}&
*+<2pt>[F-:<20pt>]{Q_1-1,Q_2-1}\ar@/^1pc/[r]^{\lambda_{2}}\ar@/^1pc/[d]^>>>>>>{\lambda_{1}}&
*+<2pt>[F-:<20pt>]{Q_1-1,Q_2-2}\ar@/^1pc/[d]^>>>>>>{\lambda_{1}}&...&
*+<2pt>[F-:<20pt>]{Q_1-1,1}\ar@/^1pc/[r]^{\lambda_{2}}\ar@/^1pc/[d]^>>>>>>{\lambda_{1}}&
*+<2pt>[F-:<20pt>]{Q_1-1,0}\ar@/^1pc/[d]^>>>>>>{s_1}\\
*+<2pt>[F-:<20pt>]{Q_1-2,Q_2}\ar@/^1pc/[r]^{\lambda_{2}}&
*+<2pt>[F-:<20pt>]{Q_1-2,Q_2-1}\ar@/^1pc/[r]^{\lambda_{2}}&
*+<2pt>[F-:<20pt>]{Q_1-2,Q_2-2}&...&
*+<2pt>[F-:<20pt>]{Q_1-2,1}\ar@/^1pc/[r]^{\lambda_{2}}&
*+<2pt>[F-:<20pt>]{Q_1-2,0}\\
\vdots&\vdots&\vdots&&\vdots&\vdots\\
*+<2pt>[F-:<20pt>]{1,Q_2}\ar@/^1pc/[r]^{\lambda_{2}}\ar@/^1pc/[d]^>>>>>>{\lambda_1}&
*+<2pt>[F-:<20pt>]{1,Q_2-1}\ar@/^1pc/[r]^{\lambda_{2}}\ar@/^1pc/[d]^>>>>>>{\lambda_{1}}&
*+<2pt>[F-:<20pt>]{1,Q_2-2}\ar@/^1pc/[d]^>>>>>>{\lambda_{1}}&...&
*+<2pt>[F-:<20pt>]{1,1}\ar@/^1pc/[r]^<<<<<<{\lambda_{2}}\ar@/^1pc/[d]^>>>>>>{\lambda_{1}}&
*+<2pt>[F-:<20pt>]{1,0}\ar@/^1pc/[d]^>>>>>>{s_1}\\
*+<2pt>[F-:<20pt>]{0,Q_2}\ar@/^1pc/[r]^<<<<<<<{s_2}&
*+<2pt>[F-:<20pt>]{0,Q_2-1}\ar@/^1pc/[r]^<<<<<<<{s_2}&
*+<2pt>[F-:<20pt>]{0,Q_2-2}&...&
*+<2pt>[F-:<20pt>]{0,1}\ar@/^1pc/[r]^<<<<<<<{s_2}&
*+<2pt>[F-:<20pt>]{0,0}}}
\vspace{0.3cm}
\scriptsize{}
\normalsize
\end{center}
\caption{Transition Rate Diagram Under Fixed Replenishment Time}%
\label{fixedTrates}%
\end{figure}
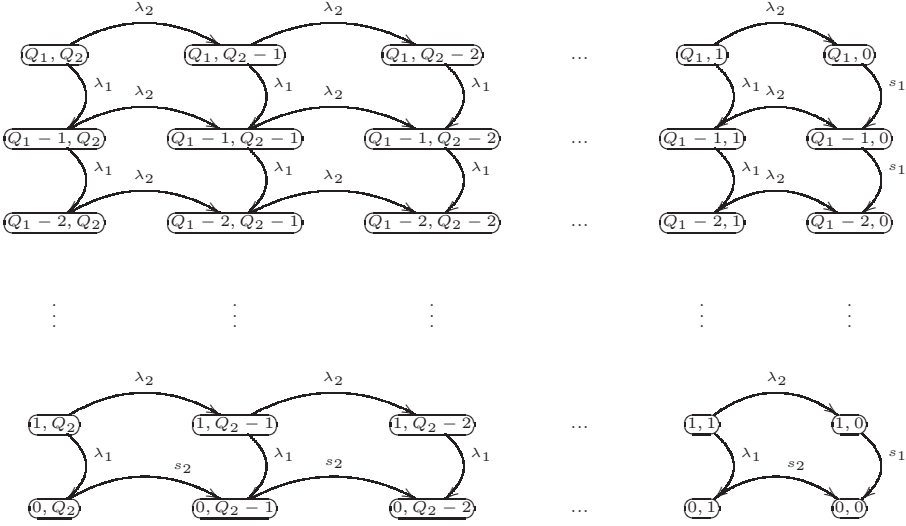

Under this model, the inventory levels at the end of a cycle correspond to the terminal state $n_1(T), n_2(T)$, thus, in order to evaluate the profit function in \eqref{profdef1}, we need to identify the transient distribution at time $T$. The explicit formula and the analysis of the profit function are derived in Section \ref{modelfixedT}.

Under the second approach, the duration of the period is exponentially distributed with parameter $\mu$. Let the inventory state at time $t$ be denoted by $(m_1(t),m_2(t))$. Then the stochastic process $\{(m_1(t),m_2(t)),t \geq 0\}$ is a continuous--time Markov chain with state space $S=\{ (i_1,i_2):i_1 \in\{ 0,1,\ldots,Q_1 \},i_2\in\{0,1,\ldots,Q_2 \} \}$. The transition rate diagram is identical to that in Figure \ref{fixedTrates}, with the addition of a transition rate equal to $\mu$ from each state to state $(Q_1, Q_2)$.
It is easy to see that the Markov chain in this model is positive recurrent, thus there exists a limiting distribution as $t\to \infty$. Furthermore, since the replenishments arrive according to a Poisson process independent of demand arrivals, it follows from the PASTA property that for large $t$ the probability distribution of the inventory level at a replenishment epoch is the limiting distribution of the Markov chain. Therefore, the value of $E_{Q_1, Q_2} (m_i), i=1,2$ in \eqref{profdef2} can be computed from the limiting distribution. The explicit formula and the analysis of the profit function are also derived in Section \ref{modelexpo}.

\section{Fixed Replenishment Time}\label{modelfixedT}

In this section we compute the transient distribution of the Markov chain $\{(n_1(t),n_2(t)):t\in[0,T]\}$ which represents the inventory levels during a cycle of fixed length $T$.  This allows to derive an analytic expression for the expected profit per period as a function of $Q_1$ and $Q_2$. Although this expression is complicated, it lends itself to numerical computations. Furthermore, we prove that the profit function is submodular, which allows for a simplification in the computation of the optimal order quantities.

We first consider the transient distribution. For notational simplicity we suppress the dependence on $(Q_1, Q_2)$ and define
\begin{equation}
p_{(j_1,j_2)}(t)=P(n_1(t)=j_1,n_2(t)=j_2),
j_1=0,\ldots,Q_1 , j_2= 0,\ldots,Q_2, t\in[0,T],
\nonumber
\end{equation}
as the transient probability function of the inventory process, given initial state  $n_1(0)=Q_1, n_2(0)=Q_2$.

To compute the transient distribution we apply uniformization and define an equivalent Markov chain where all transition times are exponentially distributed with a common rate and fictitious transitions back to the same state are allowed.
The definition of the uniformized process as well as the detailed derivations are included in the appendix. 

Let $s=\lambda_1+\lambda_2, s_1=\lambda_1+\lambda_2 p_{21}, s_2=\lambda_2+\lambda_1 p_{21}$ denote the total sales rate when both products, only product 1, or only product 2 are available, respectively. The explicit form of the transient probability distribution is derived in Theorem \ref{prop:transient_distr}. The proof is in the appendix. 

\begin{thm}\label{prop:transient_distr}
The transient probability function $p_{(j_1,j_2)}(t)$,  is equal to:
\small
\begin{eqnarray}
p_{(j_1,j_2)}(T)&=&e^{-st}\cdot\frac{(st)^{Q_1+Q_2-j_1-j_2}}{(Q_1+Q_2-j_1-j_2)!}\cdot\left(
       \begin{array}{c}
          Q_1+Q_2-j_1-j_2 \\
          Q_1-j_1 \\
       \end{array}
       \right)\cdot\nonumber\\
&&\left(\frac{\lambda_1}{s}\right)^{Q_1-j_1}\cdot\left(\frac{\lambda_2}{s}\right)^{Q_2-j_2}, \ j_1,j_2>0, \label{3.13}\\
 & & \nonumber \\
p_{(j_1,0)}(t)&=&\sum_{k=Q_1+Q_2-j_1}^{\infty}e^{-st}\cdot\frac{(st)^{k}}{k!}
                             \sum_{l=Q_2}^{Q_2+Q_1-j_1}\left(
                                                                                                   \begin{array}{c}
                                                                                                     l-1 \\
                                                                                                     Q_2-1 \\
                                                                                                   \end{array}
         \right) \cdot \nonumber\\
&&\left(\frac{\lambda_2}{s}\right)^{Q_2}\cdot\left(\frac{\lambda_1}{s}\right)^{l-Q_2} \cdot \left(
                                                                                              \begin{array}{c}
                                                                                                k-l \\
                                                                                                Q_1+Q_2-j_1-l \\
                                                                                              \end{array}                                                                               \right)\cdot\nonumber\\
&&\left(\frac{s_1}{s}\right)^{Q_1+Q_2-j_1-l}\cdot\left(\frac{s-s_1}{s}\right)^{k-Q_1-Q_2+j_1}, \ j_1>0, \label{3.14}\\
 & & \nonumber \\
p_{(0,j_2)}(t)&=&\sum_{k=Q_1+Q_2-j_2}^{\infty}e^{-st}\cdot\frac{(st)^{k}}{k!}
                             \sum_{l=Q_2}^{Q_2+Q_1-j_2}\left(
                                                                                                   \begin{array}{c}
                                                                                                     l-1 \\
                                                                                                     Q_1-1 \\
                                                                                                   \end{array}
                                                                                                 \right)\cdot\nonumber\\
&&\left(\frac{\lambda_1}{s}\right)^{Q_1}\cdot\left(\frac{\lambda_2}{s}\right)^{l-Q_1} \cdot \left(
                                                                                              \begin{array}{c}
                                                                                                k-l \\
                                                                                                Q_1+Q_2-j_2-l \\
                                                                                              \end{array}
                                                                                            \right)\cdot\nonumber\\
&&\left(\frac{s_2}{s}\right)^{Q_1+Q_2-j_2-l}\cdot\left(\frac{s-s_2}{s}\right)^{k-Q_1-Q_2+j_2}, \ j_2>0, \label{3.15}\\
 & & \nonumber \\
p_{(0,0)}(t)&=&1-\sum_{(j_1,j_2) \neq (0,0)}p_{(j_1,j_2)}(t). \label{3.16}
\end{eqnarray}
\normalsize
\end{thm}

Setting $t=T$ in the above expressions, we obtain the joint distribution of the inventory levels at the end of a replenishment cycle. Therefore the expected profit per unit time $\pi^{(1)}(Q_1, Q_2)$ can now be computed from (\ref{profdef1}) with
\begin{equation}
E_{Q_1, Q_2} (n_1)=\sum_{j_1=0}^{Q_1}j_1\sum_{j_2=0}^{Q_2}p_{(j_1,j_2)}(T)\nonumber
\end{equation}
and
\begin{equation}
E_{Q_1, Q_2} (n_1)=\sum_{j_2=0}^{Q_2}j_2\sum_{j_1=0}^{Q_1}p_{(j_1,j_2)}(T).\nonumber
\end{equation}

Although the resulting expression for the profit function is complicated and not immediately amenable to showing analytical properties, it can be used to  determine the optimal stocking policy numerically. Several numerical studies which lead to interesting managerial insights are performed in Section \ref{sec:numerical-analysis}.

In the next theorem we establish that the  profit function is submodular. This property is useful both theoretically and computationally, since it implies that the the optimal order quantity for either product while keeping the order quantity of the other product fixed is nonincreasing in the order quantity of the other product. This is intuitive, since by increasing the stocking level of one product, the shortage risk of the other one is also alleviated, because of the substitution effect, thus making high stocking levels less necessary.

\begin{thm} \label{thm:submodularity_fixed}
The expected profit $\pi^{(1)}(Q_1,Q_2)$ is submodular in $Q_1, Q_2$.
\end{thm}

\begin{proof}
We must show the inequality
\small
\begin{equation}
\pi^{(1)}(Q_1+1,Q_2+1)-\pi^{(1)}(Q_1+1,Q_2)-\pi^{(1)}(Q_1,Q_2+1)+\pi^{(1)}(Q_1,Q_2)\leq 0.
\label{3.20}
\end{equation}
\normalsize

Since $r_i,h_i\geq0$, $i=1,2$, it follows from  (\ref{profdef1}) that  it suffices to show the opposite inequality for the expected unsold quantities, i.e.,
\small
\begin{equation}
E_{(Q_1+1,Q_2+1)}(n_i)-E_{(Q_1+1,Q_2)}(n_i)\geq
E_{(Q_1,Q_2+1)}(n_i)-E_{Q_1, Q_2} (n_i), i=1,2.
\label{3.21}
\end{equation}
\normalsize

We will prove a sample-path-wise version of (\ref{3.21}), using coupling arguments
\begin{equation}
n_{i}^{(Q_1+1,Q_2+1)}(t) - n_{i}^{(Q_1+1,Q_2)}(t) \geq n_{i}^{(Q_1,Q_2+1)}(t) - n_i^{(Q_1,Q_2)}(t), i=1,2
\label{3.21a}
\end{equation}

We consider four versions of the process $\{n_1(t), n_2(t), t\in [0,T]\}$ referred to as systems, which run in parallel and are coupled as follows: At $t=0$ system A is in state $(Q_1+1,Q_2+1)$, system B in state $(Q_1+1,Q_2)$, system C in state $(Q_1,Q_2+1)$ and system D in state $(Q_1,Q_2)$. These four systems are coupled via demand arrivals and substitutions, i.e., we assume that the same stream of demand arrivals and substitution requests occurs in all four.
For each coupled system we denote the inventory level with an exponent corresponding to the initial state, e.g. the inventory levels under system A are denoted by $n_{i}^{(Q_1+1,Q_2+1)}(t)$, etc.

Consider product 1 first. We can easily see that for every possible event sequence the inventory levels of product 1 in systems A and B are always identical, and the same is true for the levels in systems C and D (i.e., $n_{1}^{(Q_1+1,Q_2+1)}(t)=n_{1}^{(Q_1+1,Q_2)}(t)$ and $n_{1}^{(Q_1,Q_2+1)}(t)=n_1^{(Q_1,Q_2)}(t)$, thus (\ref{3.21a}) holds for $i=1$) for all $t$ until either $t=T$ or one of the following two events occurs:

\noindent {\bf (i)}  Assume that at time instant $t_0 < T$  system A is in state $(i,1),\ i > 1$ and a demand for product 2 occurs which would be willing to substitute product 2 with product 1. It follows that in system A no substitution has occurred thus far and  the total demand up to this time instant for all systems was $X_1=Q_1-i, X_2=Q_2-1$. Therefore, no substitution has occurred in systems B, C and D either, and they are in states $(i,0)$, $(i-1,1)$ and $(i-1,0)$, respectively).

At time $t_0$ in systems A and B a unit of product 2 is sold, while in C and D a unit of product 1 is sold due to substitution, thus the new states in A, B, C, D  are $(i,0), (i-1,0), (i-1, 0)$ and $(i-2,0)$, respectively. This relationship between states, i.e., $n_{1}^{(Q_1+1,Q_2+1)}(t)=1+n_{1}^{(Q_1+1,Q_2)}(t)$ and $n_{1}^{(Q_1+1,Q_2)}(t)=n_{1}^{(Q_1,Q_2+1)}(t)=1+n_1^{(Q_1,Q_2)}(t)$ will hold for all $t>t_0$, until either $t=T$ or until systems A, B, C and D reach states $(2,0)$, $(1,0)$, $(1,0)$ and $(0,0)$, respectively. Then, if a demand of product 1 comes or a demand of product 2 comes and substitutes, the systems will go in states $(1,0)$, $(0,0)$, $(0,0)$ and $(0,0)$, respectively. Subsequently, the four systems will stay in these states until a demand of product 1 comes or a demand
of product 2 comes and substitutes, in which case  all will be absorbed in state $(0,0)$ until $t=T$ and a replenishment occurs.
It is easy to see that (\ref{3.21a}) holds for $i=1$ during the entire interval $[t_0,T]$.

\noindent {\bf (ii)} Now assume that at time $t_0<t$  system A is in state $(1,1)$ and a demand of product 2 occurs which (for systems B and D) substitutes product 2 with product 1.
Similarly it follows that at $t_0$  systems B, C and D are in states $(1,0)$, $(0,1)$ and $(0,0)$, respectively. At this time instant  the four
systems will directly reach states $(1,0)$, $(0,0)$, $(0,0)$ and $(0,0)$ and continue from that point on as described above.
It follows that (\ref{3.21a}) holds for $i=1$ during $[t_0,T]$ for this case as well.

Using a symmetric line of arguments it can be shown that (\ref{3.21a}) also holds for $i=2$ during $[0,T]$. Equation (\ref{3.21}) follows from equation (\ref{3.21a}) and the proof is complete.
\end{proof}

\section{Random Replenishment Time} \label{modelexpo}

In this section we derive the stationary distribution of the Markov chain $\{(m_1(t),m_2(t)),t\in\mathbb{R}\}$ which represents the inventory levels when the duration of a cycle is exponentially distributed. In analogy with the fixed replenishment time case, we also obtain an explicit expression for the expected profit per period and prove that it is a submodular function of $Q_1, Q_2$.

We first consider the stationary distribution in Theorem \ref{stat-distr-random} below. The proof is based on deriving appropriate recursive forms of the steady state equations and is included in the appendix.

\begin{thm}\label{stat-distr-random}
  The stationary distribution
  $(\pi_{(i_1,i_2)}^{(Q_1,Q_2)}:(i_1,i_2)\in S)$ of the continuous
  time Markov chain $\{(m_1(t),m_2(t)),t\in\mathbb{R}\}$ is given by
  the formulas:

  \small
  \begin{equation}
    \pi_{(Q_1-i_1,Q_2-i_2)}^{(Q_1,Q_2)}=\left(
      \begin{array}{c}
        i_1+i_2 \\
        i_1 \\
      \end{array}
    \right)
    q_{1}^{i_1}q_{2}^{i_2}q_{3}, \ i_1=0,1,\ldots,Q_1-1, \ i_2=0,1,\ldots,Q_2-1,\label{4.1}
  \end{equation}
  \normalsize

  \small
  \begin{equation}
    \pi_{(Q_1-i_1,0)}^{(Q_1,Q_2)}=B_2q_{2}^{Q_2-1}q_{3}\sum_{k=0}^{i_1}A_{1}^{i_1-k}q_{1}^{k}\left(
      \begin{array}{c}
        Q_2+k-1 \\
        k \\
      \end{array}
    \right)
    , \ i_1=0,1,\ldots,Q_1-1,\label{4.2}
  \end{equation}
  \normalsize

  \small
  \begin{equation}
    \pi_{(0,Q_2-i_2)}^{(Q_1,Q_2)}=B_1q_{1}^{Q_1-1}q_{3}\sum_{k=0}^{i_2}A_{2}^{i_2-k}q_{2}^{k}\left(
      \begin{array}{c}
        Q_1+k-1 \\
        k \\
      \end{array}
    \right)
    , \ i_2=0,1,\ldots,Q_2-1,\label{4.3}
  \end{equation}
  \normalsize

\begin{eqnarray}
  \pi_{(0,0)}^{(Q_1,Q_2)}&=&\frac{s_1}{\mu}B_2q_{2}^{Q_2-1}q_{3}\sum_{k=0}^{Q_1-1}A_{1}^{Q_1-1-k}q_{1}^{k}\left(
    \begin{array}{c}
      Q_2+k-1 \\
      k \\
    \end{array}
  \right)\nonumber\\
  &&+\frac{s_2}{\mu}B_1q_{1}^{Q_1-1}q_{3}\sum_{k=0}^{Q_2-1}A_{2}^{Q_2-1-k}q_{2}^{k}\left(
    \begin{array}{c}
      Q_1+k-1 \\
      k \\
    \end{array}
  \right),\label{4.4}
\end{eqnarray}

where
\begin{equation}
  q_i=\frac{\lambda_i}{s+\mu}, \ i=1,2,\label{4.5}
\end{equation}
\begin{equation}
  q_3=\frac{\mu}{s+\mu},\label{4.6}
\end{equation}
\begin{equation}
  A_1=\frac{s_1}{s_1+\mu},\label{4.7}
\end{equation}
\begin{equation}
  A_2=\frac{s_2}{s_2+\mu},\label{4.8}
\end{equation}
\begin{equation}
  B_1=\frac{\lambda_1}{s_2+\mu},\label{4.9}
\end{equation}
and
\begin{equation}
  B_2=\frac{\lambda_2}{s_1+\mu}.\label{4.10}
\end{equation}
\end{thm}

Based on the stationary distribution, the expected profit per unit time $\pi^{(2)}(Q_1, Q_2)$ can now be computed from (\ref{profdef2}) with

\begin{equation}
Em_1=\sum_{i_1=0}^{Q_1}(Q_1-i_1)\sum_{i_2=0}^{Q_2}\pi_{(Q_1-i_1,Q_2-i_2)}^{(Q_1,Q_2)},\nonumber
\end{equation}
and
\begin{equation}
Em_2=\sum_{i_2=0}^{Q_2}(Q_2-i_2)\sum_{i_1=0}^{Q_1}\pi_{(Q_1-i_1,Q_2-i_2)}^{(Q_1,Q_2)}.\nonumber
\end{equation}

Finally, we again prove the submodularity of the expected profit per period for the exponentially case in the following Theorem.

\begin{thm}
$\pi^{(2)}(Q_1,Q_2)$ is submodular.
\end{thm}

\begin{proof}
In order to prove that $\pi^{(2)}(Q_1,Q_2)$ is submodular we have to prove the inequality
\small
\begin{equation}
\pi^{(2)}(Q_1+1,Q_2+1)-\pi^{(2)}(Q_1+1,Q_2)-\pi^{(2)}(Q_1,Q_2+1)+\pi^{(2)}(Q_1,Q_2)\leq 0.\label{4.25}
\end{equation}
\normalsize

Since $r_i,h_i\geq0$, $i=1,2$, in order to prove \eqref{4.25} it suffices to show that
\small
\begin{equation}
Em_{i}^{(Q_1+1,Q_2+1)}-Em_{i}^{(Q_1+1,Q_2)}\geq
Em_{i}^{(Q_1,Q_2+1)}-Em_i^{(Q_1,Q_2)}, i=1,2,\label{4.26}
\end{equation}
\normalsize
and we will prove again with sample-path arguments the following equation
\small
\begin{equation}
m_{i}^{(Q_1+1,Q_2+1)}-m_{i}^{(Q_1+1,Q_2)}\geq
m_{i}^{(Q_1,Q_2+1)}-m_i^{(Q_1,Q_2)}, i=1,2.\label{4.27}
\end{equation}

We follow the same approach as in Theorem \ref{thm:submodularity_fixed}, i.e., we consider four systems A, B, C, D with initial states $(Q_1+1,Q_2+1)$, $(Q_1+1,Q_2)$, $(Q_1,Q_2+1)$ and $(Q_1,Q_2)$,  respectively. These are coupled in demand arrivals, substitution requests and replenishment arrivals. Let $\tilde{T}$ be the common first replenishment time. Using exactly the same sample path arguments as in Theorem 2, it can be shown that (\ref{4.27}) holds for $i=1,2$ and $0\leq t\leq \tilde{T}$. Thus, (\ref{4.26}) holds and the proof is complete.
\end{proof}

\section{Optimization Algorithm}

In this section  we construct an algorithm in order to determine the optimal quantities $Q_1$ and $Q_2$ that will be ordered at the beginning of each period. For the fixed replenishment case we must solve the following problem

\begin{equation}
\max_{\underset{a_1 Q_1+a_2 Q_2\leq C}{Q_1,Q_2}} \pi^{(1)}(Q_1,Q_2).\nonumber
\end{equation}

For fixed $Q_1$, denote by $Q_{2}^{*}(Q_1)$ the largest value of $Q_2$ at which the $\max_{Q_2\in\{0,\ldots,\left\lfloor\frac{C- a_1 Q_1}{a_2}\right\rfloor\}} \pi^{(1)}(Q_1,Q_2)$ occurs. Thus, 

\small
\begin{eqnarray}
 Q_{2}^{*}(Q_1)&=&\max\{Q_{2}^{0}\in\{0,\ldots,\left\lfloor\frac{C- a_1 Q_1}{a_2}\right\rfloor\}:\nonumber\\ &&\max_{Q_2\in\{0,\ldots,\left\lfloor\frac{C- a_1 Q_1}{a_2}\right\rfloor\}} \pi^{(1)}(Q_1,Q_2)=\pi^{(1)}(Q_1,Q_{2}^{0})\}\nonumber.
\end{eqnarray}
\normalsize

Therefore, $\max_{\underset{a_1 Q_1+a_2 Q_2\leq C}{Q_1,Q_2}} \pi^{(1)}(Q_1,Q_2)$ can be written as

\begin{eqnarray}
\max_{\underset{a_1 Q_1+a_2 Q_2\leq C}{Q_1,Q_2}} \pi^{(1)}(Q_1,Q_2)&=&\max_{Q_1\in\{0,\ldots,\left\lfloor\frac{C}{a_1}\right\rfloor\}}\{\max_{Q_2\in\{0,\ldots,\left\lfloor\frac{C- a_1 Q_1}{a_2}\right\rfloor\}}\pi^{(1)}(Q_1,Q_2)\}\nonumber\\
 &=& \max_{Q_1\in\{0,\ldots,\left\lfloor\frac{C}{a_1}\right\rfloor\}}\pi^{(1)}(Q_1,Q_{2}^{*}(Q_1)),\nonumber
\end{eqnarray}
and the problem takes the following form

\begin{equation}
\max_{Q_1\in\{0,\ldots,\left\lfloor\frac{C}{a_1}\right\rfloor\}}\pi^{(1)}(Q_1,Q_{2}^{*}(Q_1)).\nonumber
\end{equation}

Respectively, for exponentially distributed replenishment times, we must solve the problem

\begin{equation}
\max_{Q_1\in\{0,\ldots,\left\lfloor\frac{C}{a_1}\right\rfloor\}}\pi^{(2)}(Q_1,Q_{2}^{*}(Q_1)).\nonumber
\end{equation}

For both cases, from the submodularity of $\pi^{(1)}(Q_1,Q_2)$, $\pi^{(2)}(Q_1,Q_2)$ we immediately obtain the following Lemma.

\begin{lem}
The optimal order quantity $Q_{2}^{*}(Q_1)$ is decreasing in $Q_1$.
\end{lem}

Using the monotonicity of  $Q_{2}^{*}(Q_1)$, we can construct an improved search algorithm to determine the optimal quantities $Q_1$ and $Q_2$. The algorithm has the same form for both profit functions $\pi^{(i)}, i=1,2$. 

\underline{Algorithm}
\begin{itemize}
\item Find,
\small
\begin{equation}
Q_{2}^{*}(0)=\max\{Q_{2}^{0}\in\{0,\ldots,\left\lfloor\frac{C}{a_2}\right\rfloor\}:\max_{Q_2\in\{0,\ldots, \left\lfloor\frac{C}{a_2}\right\rfloor\}}\pi^{(i)}(0,Q_2)= \pi^{(i)}(0,Q_{2}^{0}) \}.\nonumber
\end{equation}
\normalsize
\item For $j=1:\left\lfloor\frac{C}{a_1}\right\rfloor$ find,
\begin{eqnarray}
Q_{2}^{*}(j)&=&\max\{Q_{2}^{0}\in\{0,\ldots,\min(Q_{2}^{*}(j-1),\left\lfloor\frac{C-a_1 j}{a_2}\right\rfloor)\}:\nonumber\\
&&\hspace{1cm}\max_{Q_2}\pi^{(i)}(j,Q_2)= \pi^{(i)}(j,Q_{2}^{0}) \}.\nonumber
\end{eqnarray}
\item Find, $Q_{1}^{*}$ such that
\begin{equation}
\pi^{(i)}(Q_{1}^{*},Q_{2}^{*}(Q_{1}^{*}))=\max_{Q_1\in\{0,\ldots,\left\lfloor\frac{C}{a_1}\right\rfloor\}}\pi^{(i)}(Q_1,Q_{2}^{*}(Q_1)).\nonumber
\end{equation}
\end{itemize}

\section{Computational Analysis}\label{sec:numerical-analysis}
In this section we employ the algorithms developed previously in order to obtain insights on the effect of substitution on the ordering policies and the profit.
We are specifically interested in two general questions. First, how are the ordering policies and the expected profit per period affected when there is substitution between the two products compared to the case when no substitution is possible, and second how does the randomness of the replenishment time affect the ordering policies and the profit.

We first analyze  the effect of substitution, under constant replenishment time. We consider two products for which the Poisson demand rates are equal to $\lambda_1=\lambda_2=20$ units per unit time and normalize the length of the replenishment time to $T=1$.
We also assume $a_1=a_2=1$, which reduces the ordering constraint to $Q_1+Q_2=C$, thus $C$ represents the common storage capacity. 
 Regarding the economic parameters, we distinguish three representative cases corresponding to different product types. Specifically, if the storage capacity is practically infinite and there is no substitution, then the ordering problem is decomposed and for each product the newsvendor rule applies: the optimal order quantity of product $j$ is the first so that the service level is at least equal to the critical ratio $R_j=\frac{p_j-c_j}{p_j+h_j}$. Based on this observation we consider the following scenarios: (i) both products have a relatively high critical ratio $R_1=R_2=0.8$, (ii) one product has high and the other a low critical ratio $R_1=0.8, R_2=0.4$ and (iii) both critical ratios are low $R_1=R_2=0.4$. For each of the three scenarios we compute the optimal ordering policy and the optimal expected profit per period for varying capacity levels under two options: first when no substitution between the two products is possible, and second when there is two-way substitution with $p_{12}=p_{21}=0.4$. The results are presented in Figures \ref{fig:Tscen1}, \ref{fig:Tscen2} and \ref{fig:Tscen3}, respectively.

Several insights can be obtained from these results. First, the profit curves show that, as expected, the newsvendor is better off when there is substitutability between the two products. Second, this is generally achieved by changing the mix of products ordered, in favor of the most profitable of the two.

Specifically, under scenario 1 where both products have a relatively high order quantity, the benefit of substitution is small to nonexistent when the capacity is either very low or very high. In the first case, the order quantities are much smaller than those needed, and as a result both products run out together and fast, without giving substitution a chance to occur. In this case the order quantities are also the same between the two possibilities.
 On the other hand when the capacity is very high, substitution also offers a rather small benefit. It is possible to order a lower total quantity of both products without affecting the service level. This is so because substitution essentially allows to partially  take advantage of the pooling effect between the two independent demands, thus decreasing the downside risk of unsold products.
This beneficial behavior is even more pronounced when the capacity is low enough to be a restrictive factor for the profit, but high enough to allow substitution to take effect. In this intermediate range the relative  profit increase due to substitution is highest.

Under scenario 2, where both products have a relatively low order quantity, we obtain qualitatively similar insights, as under scenario 1. On the other hand, when the two products differ significantly in terms of their critical ratios, which is the case under scenario 3, then the effect of substitution can be quite pronounced to the extent that the product with the lower critical ratio may be completely dropped and its demand satisfied partially through substitution of the other product. In this case, the benefit of substitution in terms of profitability is persistent even when the storage capacity is nonbinding.

\begin{figure}
  \centering
 \includegraphics{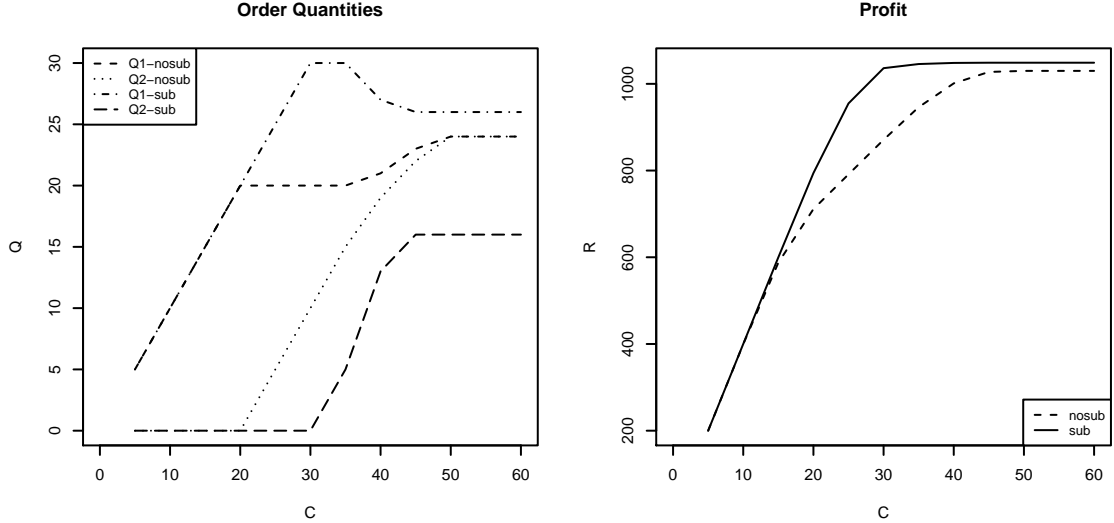}
  \caption{Effect of Substitution, Scenario 1:  $\lambda_1=\lambda_2=20, T=1, p_1=50, c_1=10, h_1=0, p_2=20, c_2=4, h_2=0, p_{12}=p_{21} =0 (\mbox{nosub}),  0.4 (\mbox{sub})$ }
  \label{fig:Tscen1}
\end{figure}

\begin{figure}
  \centering
 \includegraphics{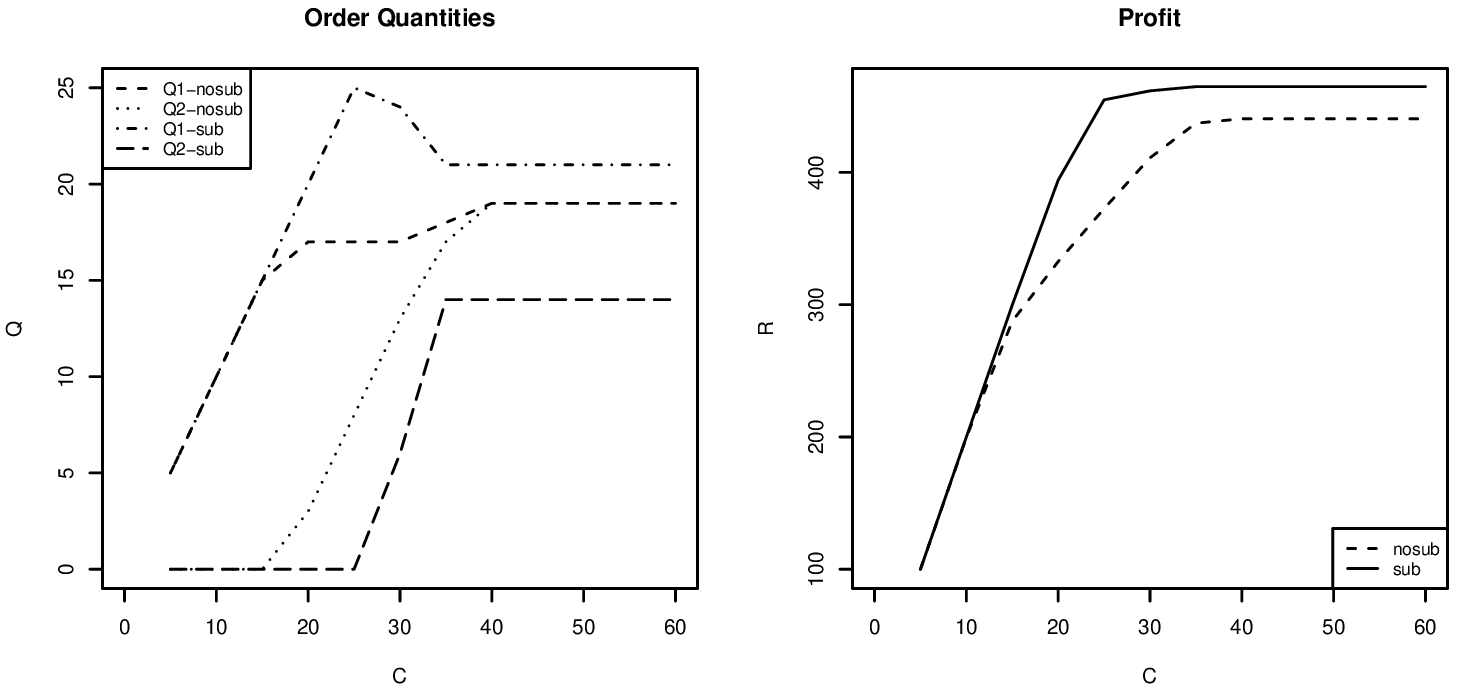}
\caption{Effect of Substitution, Scenario 2:  $\lambda_1=\lambda_2=20, T=1, p_1=50, c_1=30, h_1=0, p_2=20, c_2=12, h_2=0, p_{12}=p_{21} =0 (\mbox{nosub}),  0.4 (\mbox{sub})$ }
  \label{fig:Tscen2}
\end{figure}

\begin{figure}
  \centering
 \includegraphics{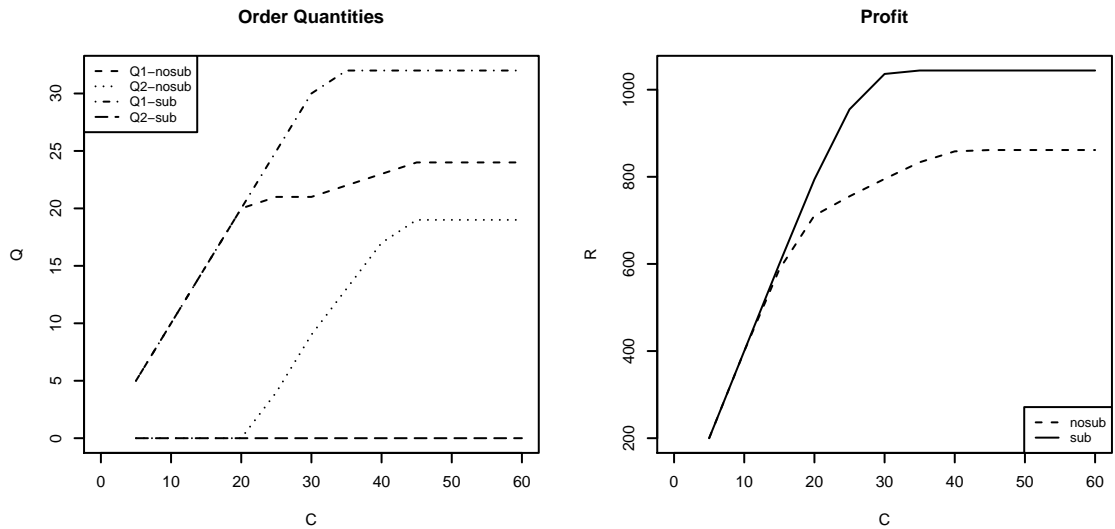}
\caption{Effect of Substitution, Scenario 3:  $\lambda_1=\lambda_2=20, T=1, p_1=50, c_1=10, h_1=0, p_2=20, c_2=12, h_2=0, p_{12}=p_{21} =0 (\mbox{nosub}),  0.4 (\mbox{sub})$ }
  \label{fig:Tscen3}
\end{figure}

In the second set of our numerical experiments, we aim to assess the impact of the variability in the replenishment time on the ordering policy and the profits under substitution. To this end we first performed a set of experiments analogous to scenarios 1-3 above, but now with the replenishment time exponentially distributed with mean 1.  The behavior of the ordering policy and the effect of substitution were similar to those under fixed replenishment time, although the expected profit was uniformly lower, indicating that the variability in replenishment time has a detrimental effect. To further analyze the effect,  we show in Figures \ref{fig:repscen1}, \ref{fig:repscen2} and \ref{fig:repscen3} the comparison of the ordering policies and profits under substitution between the cases of fixed and exponentially distributed replenishment times. From the results it is verified that when the replenishment time is random there is a significant decrease in the profit. Regarding the effect on the ordering policy, this depends on the product types. Specifically, under scenario 1, the upside risk of shortages is more costly than that of unsold products. Therefore, in the random case the more detrimental event is that the replenishment time will be large and as a result there will be high shortages. To mitigate this possibility the order quantities are significantly higher than those in the fixed case. On the other hand, under scenario 2, where the two critical ratios are low, the downside risk is more important and as a result the ordering policy gravitates towards lower quantities, to mitigate the possibility of short replenishment times and large quantities of unsold products. Under scenario 3, the results are similar, since product 1 with the high critical ratio is ordered at higher quantities under the random case, while the second product with the low critical ratio is not ordered at all in both cases, as already discussed in the first experiment. 

\begin{figure}
  \centering
 \includegraphics{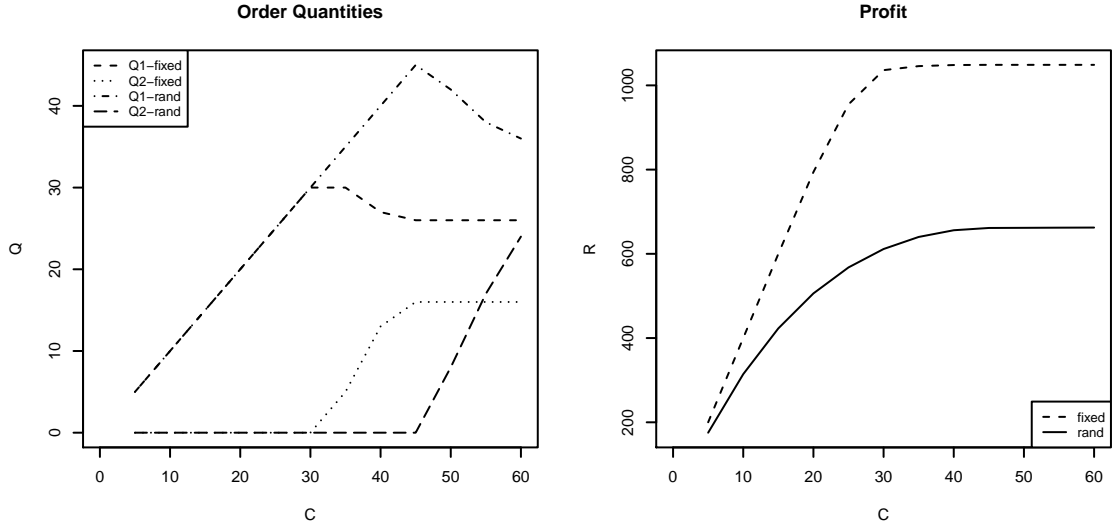}
  \caption{Effect of Random Replenishment Time, Scenario 1:  $\lambda_1=\lambda_2=20, T=1, p_1=50, c_1=10, h_1=0, p_2=20, c_2=4, h_2=0, p_{12}=p_{21} =0.4, T=1 (\mbox{fixed)}/ T\sim Exp(1) (\mbox{rand})$ }
  \label{fig:repscen1}
\end{figure}

\begin{figure}
  \centering
 \includegraphics{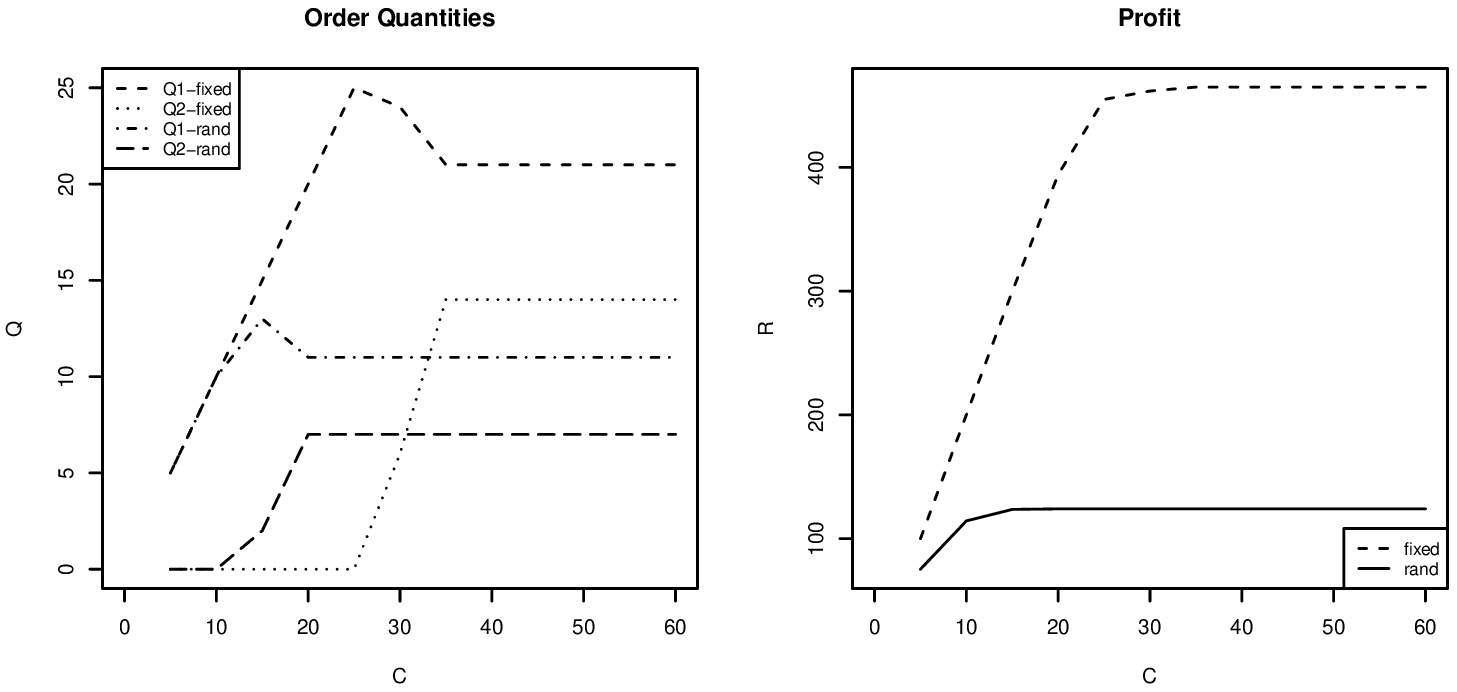}
\caption{Effect of Substitution, Scenario 2:  $\lambda_1=\lambda_2=20, T=1, p_1=50, c_1=30, h_1=0, p_2=20, c_2=12, h_2=0, p_{12}=p_{21} =0.4, T=1 (\mbox{fixed)}/ T\sim Exp(1) (\mbox{rand})$ }
  \label{fig:repscen2}
\end{figure}

\begin{figure}
  \centering
 \includegraphics{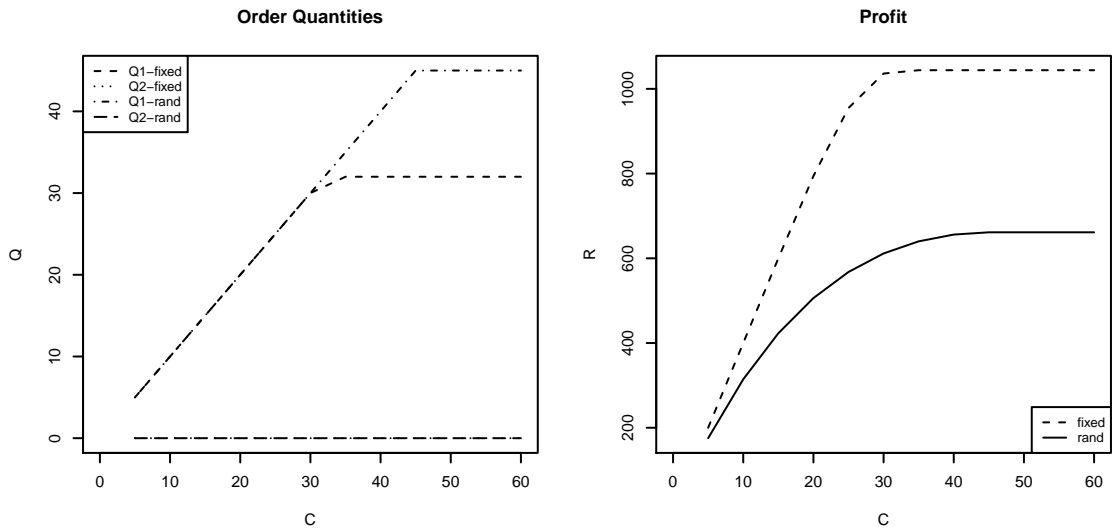}
\caption{Effect of Substitution, Scenario 3:  $\lambda_1=\lambda_2=20, T=1, p_1=50, c_1=10, h_1=0, p_2=20, c_2=12, h_2=0, p_{12}=p_{21} =0.4, T=1 (\mbox{fixed)}/ T\sim Exp(1) (\mbox{rand})$ }
  \label{fig:repscen3}
\end{figure}

\section{Conclusions and Extensions}\label{sec:extensions}

In this paper we developed a stochastic model of joint replenishment in an inventory system with two products under Poisson demand arrivals, limited storage capacity and partial two-way substitution. We have considered two cases of replenishment, periodic replenishment at fixed time intervals and random replenishment times. In both cases we defined a two-dimensional Markov Chain model for the inventory evolution and derived analytic expressions for the cost function. We established that the expected cost is submodular in the order quantities and employed this property to derive an efficient algorithm for computing the optimal policy. In computational experiments we explored the interaction between the substitution and limited capacity, and assessed the combined impact on the cost. 

This research can be extended in several directions. First, in addition to single echelon systems, substitution also affects inventory management in a supply chain framework, both in terms of centralized optimal policies, as well as in the interactions between agents and the formation of contacting mechanisms (cf. \cite{kraiselburd2004contracting}, who study the interaction between VMI policies and substitution in a single period model with stochastic demand). In particular, it would be interesting to study the formation of ordering policies under quantity discounts and product substitution in the context of the Poisson demand model developed in this paper (cf. \cite{katehakis-smit2012} for analysis of $(Q,R)$ ordering policies under Poisson demand and multiple quantity discounts). 

In the present paper we modeled limitations in the replenishment mechanism by a finite storage capacity. However in many cases inventory replenishment is limited by a finite production rate. When two or more products are produced by the same equipment with finite rate and possibly with significant production switching costs, immediate joint replenishment is not generally possible. In such situations the presence of substitution between products may mitigate the shortage effects and must be taken into account when designing ordering and production policies. A two-product system with product substitution could be defined as an extension of both our paper and \cite{shi2014production}, who analyze the single product case under finite replenishment rate and Poisson demand arrivals.

Finally, in our model we have assumed that the demand process is completely known both in terms of the arrival rates and the substitution probabilities. If this is not the case, there is a need to estimate the corresponding parameters, thus giving rise to the extension into adaptive estimation and ordering policies. The problem of parameter estimation in the present context is not straightforward. Indeed, if, for example, only sales and not actual demand is observed and at the end of a period one of the two products is out of stock, this may be due to a high demand rate of this product or alternatively of a high demand rate and a high substitution probability of the other product or both. The design of effective estimation methods combined with adaptive ordering policies is currently work in progress. 

\medskip

\centerline{\bf Acknowledgment}
This research has been co-financed by the European Union (European Social Fund - ESF) and Greek national funds through the Operational Program ''Education and Lifelong Learning'' of the National Strategic Reference Framework (NSRF)-Research Funding Program: Thales-Athens University of Economics and Business-New Methods in the Analysis of Market Competition: Oligopoly, Networks and Regulation.

\bibliographystyle{abbrvnat}


\newpage
\appendix
\section{Appendix}

\subsection{Proof of Theorem 1 (Transient distribution)}

Firstly, we define the generator matrix and using uniformization method we express the formulas of transient distribution.

As we mentioned before, in this case the duration of a period is $T$. Fix a replenishment policy $(Q_1,Q_2)$. The state of the vendor at time $t$ is represented by a pair $(n_1(t),n_2(t))$, where $n_{i}(t)$ denotes the inventory of product $i$ at time $t$, $i=1,2$, $t\in[0,T]$. The stochastic process $\{(n_1(t),n_2(t)):t\in[0,T]\}$ is a continuous time Markov chain  and infinitesimal generator matrix

\begin{equation}
\mathbb{Q}^{(1)}=\left(
                   \begin{array}{cccccc}
                     \mathbb{Q}_{0}^{(1)} & \mathbb{O} & \mathbb{O} & \ldots & \mathbb{O} & \mathbb{O} \\
                     \mathbb{Q}_{1}^{(1)} & \mathbb{Q}_{2}^{(1)} & \mathbb{O} & \ldots & \mathbb{O} & \mathbb{O} \\
                     \mathbb{O} & \mathbb{Q}_{1}^{(1)} & \mathbb{Q}_{2}^{(1)} & \ldots & \mathbb{O} & \mathbb{O} \\
                     \vdots & \vdots & \vdots & \ddots & \vdots & \vdots \\
                     \mathbb{O} & \mathbb{O} & \mathbb{O} & \ldots & \mathbb{Q}_{2}^{(1)} & \mathbb{O} \\
                     \mathbb{O} & \mathbb{O} & \mathbb{O} & \ldots & \mathbb{Q}_{1}^{(1)} & \mathbb{Q}_{2}^{(1)} \\
                   \end{array}
                 \right),\label{2.1}
\end{equation}
$\mathbb{Q}^{(1)}\in M_{Q_1 +1 \times Q_2 +1,Q_1 +1\times Q_2 +1}$, where

\begin{equation}
\mathbb{Q}_{0}^{(1)}=\left(
                       \begin{array}{cccccc}
                         0 & 0 & 0 & \ldots & 0 & 0 \\
                         s_2 & -s_2 & 0 & \ldots & 0 & 0 \\
                         0 & s_2 & -s_2 & \ldots & 0 & 0 \\
                         \vdots & \vdots & \vdots & \ddots & \vdots & \vdots \\
                         0 & 0 & 0 & \ldots & -s_2 & 0 \\
                         0 & 0 & 0 & \ldots & s_2 & -s_2 \\
                       \end{array}
                     \right)
,\label{2.2}
\end{equation}
$\mathbb{Q}_{0}^{(1)}\in M_{Q_2 +1,Q_2 +1}$, where $s_2=\lambda_2 +\lambda_1 p_{12}$ and the rows and columns correspond to states $(0,j)$, $j=0,1,\ldots,Q_2$,

\begin{equation}
\mathbb{Q}_{1}^{(1)}=\left(
                       \begin{array}{cccccc}
                         s_1 & 0 & 0 & \ldots & 0 & 0 \\
                         0 & \lambda_1 & 0 & \ldots & 0 & 0 \\
                         0 & 0 & \lambda_1 & \ldots & 0 & 0 \\
                         \vdots & \vdots & \vdots & \ddots & \vdots & \vdots \\
                         0 & 0 & 0 & \ldots & \lambda_1 & 0 \\
                         0 & 0 & 0 & \ldots & 0 & \lambda_1 \\
                       \end{array}
                     \right)
,\label{2.3}
\end{equation}
$\mathbb{Q}_{1}^{(1)}\in M_{Q_2 +1,Q_2 +1}$, where $s_1=\lambda_1 +\lambda_2 p_{21}$, and for fix $i$, $i=1,\ldots,Q_1$ the rows correspond to states $(i,j)$, $j=0,1,\ldots,Q_2$ and the columns correspond to states $(i-1,j)$, $j=0,1,\ldots,Q_2$,

\begin{equation}
\mathbb{Q}_{2}^{(1)}=\left(
                       \begin{array}{cccccc}
                         -s_1 & 0 & 0 & \ldots & 0 & 0 \\
                         \lambda_2 & -s & 0 & \ldots & 0 & 0 \\
                         0 & \lambda_2 & -s & \ldots & 0 & 0 \\
                         \vdots & \vdots & \vdots & \ddots & \vdots & \vdots \\
                         0 & 0 & 0 & \ldots & -s & 0 \\
                         0 & 0 & 0 & \ldots & \lambda_2 & -s \\
                       \end{array}
                     \right)
,\label{2.4}
\end{equation}
$\mathbb{Q}_{2}^{(1)}\in M_{Q_2 +1,Q_2 +1}$, where $s=\lambda_1 +\lambda_2$, and for fix $i$, $i=1,\ldots,Q_1$ the rows and the columns correspond to states $(i,j)$, $j=0,1,\ldots,Q_2$,

\begin{equation}
\mathbb{O}=\left(
                       \begin{array}{cccccc}
                         0 & 0 & 0 & \ldots & 0 & 0 \\
                         0 & 0 & 0 & \ldots & 0 & 0 \\
                         0 & 0 & 0 & \ldots & 0 & 0 \\
                         \vdots & \vdots & \vdots & \ddots & \vdots & \vdots \\
                         0 & 0 & 0 & \ldots & 0 & 0 \\
                         0 & 0 & 0 & \ldots & 0 & 0 \\
                       \end{array}
                     \right)
,\label{2.5}
\end{equation}
$\mathbb{O}\in M_{Q_2 +1,Q_2 +1}$.

Let
\begin{equation}
p(t)=[p_{(j_1,j_2)}(t)],\nonumber
\end{equation}
be the row vector that gives the transient distribution of the continuous time Markov chain $\{(n_1(t),n_2(t)):t\in[0,T]\}$, where

\begin{eqnarray}
p_{(i_1,i_2),(j_1,j_2)}(t)&=&P(n_1(t)=j_1,n_2(t)=j_2| n_1(o)=i_1,\nonumber\\
&&n_2(0)=i_2),\nonumber
\end{eqnarray}
$i_1,j_1\in\{ 0,1,2,\ldots,Q_1 \}$, $i_2,j_2\in\{ 0,1,2,\ldots,Q_2 \}$, $t\in[0,T]$, is the transition probability of the continuous time Markov chain $\{(n_1(t),n_2(t)):t\in[0,T]\}$,

\begin{equation}
\mathbb{P}(t)=[p_{(i_1,i_2),(j_1,j_2)}(t)],\nonumber
\end{equation}
is the transition matrix of the continuous time Markov chain $\{(n_1(t),$ $n_2(t)):t\in[0,T]\}$.

So, we want to determine the row vector $p(T)$. We have that

\begin{equation}
p(T)=p(0)\cdot \mathbb{P}(T).\label{3.1}
\end{equation}

At the beginning of the period $Q_1$ units of $1$ and $Q_2$ units of product $2$ are ordered. So, the initial distribution $p(0)$ is given by

\begin{equation}
p(0)=\left(
                   \begin{array}{cccccc}
                     0 & 0 & 0 & \ldots & 0 & 1 \\
                   \end{array}
                 \right)
.\label{3.2}
\end{equation}

Since the only nonzero element of  $p(0)$ is the last one, we only need to know the last row of matrix $\mathbb{P}(T)$ in order to compute $p(T)$.

We will determine those elements of $p(T)$ using uniformization method. The largest absolute value of generator matrix $\mathbb{Q}^{(1)}$ is $s$, (recall that $s=\lambda_1+\lambda_2$), so we have that

\begin{equation}
p(t)=\sum_{k=0}^{\infty}e^{-st}\frac{(st)^k}{k!}\mathbb{P}^k,\label{3.3}
\end{equation}
where matrix $\mathbb{P}\in M_{Q_1 +1 \times Q_2 +1,Q_1 +1\times Q_2 +1}$ is given by

\begin{equation}
\mathbb{P}=\frac{1}{s}\mathbb{Q}^{(1)}+\mathbb{I},\label{3.4}
\end{equation}
where matrix $\mathbb{Q}^{(1)}$ is given by \eqref{2.1} and $\mathbb{I}\in M_{Q_1 +1 \times Q_2 +1,Q_1 +1\times Q_2 +1}$ is the diagonal matrix with all diagonal elements equal to one.

Therefore, to prove the formulas we need to derive the transient distribution for the discrete time chain.

So,

\begin{equation}
\mathbb{P}=\left(
            \begin{array}{cccccc}
              \mathbb{P}_0 & \mathbb{O} & \mathbb{O} & \ldots & \mathbb{O} & \mathbb{O} \\
              \mathbb{P}_1 & \mathbb{P}_2 & \mathbb{O} & \ldots & \mathbb{O} & \mathbb{O} \\
              \mathbb{O} & \mathbb{P}_1 & \mathbb{P}_2 & \ldots & \mathbb{O} & \mathbb{O} \\
              \vdots & \vdots & \vdots & \ddots & \vdots & \vdots \\
              \mathbb{O} & \mathbb{O} & \mathbb{O} & \ldots & \mathbb{P}_2 & \mathbb{O} \\
              \mathbb{O} & \mathbb{O} & \mathbb{O} & \ldots & \mathbb{P}_1 & \mathbb{P}_2 \\
            \end{array}
          \right),\label{3.5}
\end{equation}
where

\begin{equation}
\mathbb{P}_0=\left(
            \begin{array}{cccccc}
              1 & 0 & 0 & \ldots & 0 & 0 \\
              s_2/s & (s-s_2)/s & 0 & \ldots & 0 & 0 \\
              0 & s_2/s & (s-s_2)/s & \ldots & 0 & 0 \\
              \vdots & \vdots & \vdots & \ddots & \vdots & \vdots \\
              0 & 0 & 0 & \ldots & (s-s_2)/s & 0 \\
              0 & 0 & 0 & \ldots & s_2/s & (s-s_2)/s \\
            \end{array}
          \right),\label{3.6}
\end{equation}
$\mathbb{P}_{0}\in M_{Q_2 +1,Q_2 +1}$, (recall that $s_2=\lambda_2 +\lambda_1 p_{12}$) where the rows and columns correspond to states $(0,j)$, $j=0,1,\ldots,Q_2$,

\begin{equation}
\mathbb{P}_1=\left(
            \begin{array}{cccccc}
              s_1/s & 0 & 0 & \ldots & 0 & 0 \\
              0 & \lambda_1/s & 0 & \ldots & 0 & 0 \\
              0 & 0 & \lambda_1/s & \ldots & 0 & 0 \\
              \vdots & \vdots & \vdots & \ddots & \vdots & \vdots \\
              0 & 0 & 0 & \ldots & \lambda_1/s & 0 \\
              0 & 0 & 0 & \ldots & 0 & \lambda_1/s \\
            \end{array}
          \right),\label{3.7}
\end{equation}
$\mathbb{P}_{1}\in M_{Q_2 +1,Q_2 +1}$, (recall that $s_1=\lambda_1 +\lambda_2 p_{21}$), where for fix $i$, $i=1,\ldots,Q_1$ the rows correspond to states $(i,j)$, $j=0,1,\ldots,Q_2$ and the columns correspond to states $(i-1,j)$, $j=0,1,\ldots,Q_2$,

\begin{equation}
\mathbb{P}_2=\left(
            \begin{array}{cccccc}
              s-s_1/s & 0 & 0 & \ldots & 0 & 0 \\
              \lambda_2/s & 0 & 0 & \ldots & 0 & 0 \\
              0 & \lambda_2/s & 0 & \ldots & 0 & 0 \\
              \vdots & \vdots & \vdots & \ddots & \vdots & \vdots \\
              0 & 0 & 0 & \ldots & 0 & 0 \\
              0 & 0 & 0 & \ldots & \lambda_2/s & 0 \\
            \end{array}
          \right),\label{3.8}
\end{equation}
$\mathbb{P}_{2}\in M_{Q_2 +1,Q_2 +1}$, where for fix $i$, $i=1,\ldots,Q_1$ the rows and the columns correspond to states $(i,j)$, $j=0,1,\ldots,Q_2$, and $\mathbb{O}$ is given by \eqref{2.5}.

Matrix $\mathbb{P}$ is a stochastic matrix and can be considered as the one-step transition probability matrix of a discrete time Markov chain $\{ (n_{1}^{d}(k),n_{2}^{d}(k), \\ k\in\{0,1,\ldots \}) \}$ with state space $S=\{ (i_1,i_2):i_1 \in\{ 0,\ldots,Q_1 \},i_2 \in \{ 0,\ldots,Q_2 \} \}$ and transition as shown on the following diagram.

\tiny{\xymatrix{*+<2pt>[F-:<20pt>]{Q_1,Q_2}\ar@/^1pc/[r]^{\lambda_{2}/s}\ar@/^1pc/[d]^>>>>>>{\lambda_1/s}&
*+<2pt>[F-:<20pt>]{Q_1,Q_2-1}\ar@/^1pc/[r]^{\lambda_{2}/s}\ar@/^1pc/[d]^>>>>>>{\lambda_{1}/s}&
*+<2pt>[F-:<20pt>]{Q_1,Q_2-2}\ar@/^1pc/[d]^>>>>>>{\lambda_{1}/s}&...&
*+<2pt>[F-:<20pt>]{Q_1,1}\ar@/^1pc/[r]^{\lambda_{2}/s}\ar@/^1pc/[d]^>>>>>>{\lambda_{1}/s}&
*+<2pt>[F-:<20pt>]{Q_1,0}\ar@(ur,dr)^{(s-s_1)/s}\ar@/^1pc/[d]^>>>>>>{s_1/s}\\
*+<2pt>[F-:<20pt>]{Q_1-1,Q_2}\ar@/^1pc/[r]^{\lambda_{2}/s}\ar@/^1pc/[d]^>>>>>>{\lambda_1/s}&
*+<2pt>[F-:<20pt>]{Q_1-1,Q_2-1}\ar@/^1pc/[r]^{\lambda_{2}/s}\ar@/^1pc/[d]^>>>>>>{\lambda_{1}/s}&
*+<2pt>[F-:<20pt>]{Q_1-1,Q_2-2}\ar@/^1pc/[d]^>>>>>>{\lambda_{1}/s}&...&
*+<2pt>[F-:<20pt>]{Q_1-1,1}\ar@/^1pc/[r]^{\lambda_{2}/s}\ar@/^1pc/[d]^>>>>>>{\lambda_{1/s}}&
*+<2pt>[F-:<20pt>]{Q_1-1,0}\ar@(ur,dr)^{(s-s_1)/s}\ar@/^1pc/[d]^>>>>>>{s_1/s}\\
*+<2pt>[F-:<20pt>]{Q_1-2,Q_2}\ar@/^1pc/[r]^{\lambda_{2}/s}&
*+<2pt>[F-:<20pt>]{Q_1-2,Q_2-1}\ar@/^1pc/[r]^{\lambda_{2}/s}&
*+<2pt>[F-:<20pt>]{Q_1-2,Q_2-2}&...&
*+<2pt>[F-:<20pt>]{Q_1-2,1}\ar@/^1pc/[r]^{\lambda_{2}/s}&
*+<2pt>[F-:<20pt>]{Q_1-2,0}\ar@(ur,dr)^{(s-s_1)/s}\\
\vdots&\vdots&\vdots&&\vdots&\vdots\\
*+<2pt>[F-:<20pt>]{1,Q_2}\ar@/^1pc/[r]^{\lambda_{2}/s}\ar@/^1pc/[d]^>>>>>>{\lambda_1/s}&
*+<2pt>[F-:<20pt>]{1,Q_2-1}\ar@/^1pc/[r]^{\lambda_{2}/s}\ar@/^1pc/[d]^>>>>>>{\lambda_{1}/s}&
*+<2pt>[F-:<20pt>]{1,Q_2-2}\ar@/^1pc/[d]^>>>>>>{\lambda_{1}/s}&...&
*+<2pt>[F-:<20pt>]{1,1}\ar@/^1pc/[r]^<<<<<<{\lambda_{2}/s}\ar@/^1pc/[d]^>>>>>>{\lambda_{1}/s}&
*+<2pt>[F-:<20pt>]{1,0}\ar@(ur,dr)^{(s-s_1)/s}\ar@/^1pc/[d]^>>>>>>{s_1/s}\\
*+<2pt>[F-:<20pt>]{0,Q_2}\ar@(ur,dr)^{(s-s_2)/s}\ar@/^1pc/[r]^<<<<<<<{s_2/s}&
*+<2pt>[F-:<20pt>]{0,Q_2-1}\ar@(ur,dr)^{(s-s_2)/s}\ar@/^1pc/[r]^<<<<<<<{s_2/s}&
*+<2pt>[F-:<20pt>]{0,Q_2-2}\ar@(ur,dr)^{(s-s_2)/s}&...&
*+<2pt>[F-:<20pt>]{0,1}\ar@(ur,dr)^{(s-s_2)/s}\ar@/^1pc/[r]^<<<<<<<{s_2/s}&
*+<2pt>[F-:<20pt>]{0,0}\ar@(ur,dr)^{(s-s_1)/s}}}
\vspace{0.3cm}
\scriptsize{}
\normalsize

Denoting by $p_{(i_1,i_2),(j_1,j_2)}^{d}(k)$ the k-step transition probability of this discrete time Markov chain, i.e.

\begin{equation}
p_{(i_1,i_2),(j_1,j_2)}^{d}(k)=P(n_{1}^{d}(k)=j_1,n_{2}^{d}(k)=j_2| n_{1}^{d}(o)=i_1,
n_{2}^{d}(0)=i_2),\nonumber
\end{equation}
we have that $\mathbb{P}^k=[p_{(i_1,i_2),(j_1,j_2)}^{d}(k)]$.

As we mentioned before, we only need the elements of the last row of $p(T)$. So using \eqref{3.3}, we only need the elements of the last row of $\mathbb{P}^k$. Thus, transition probabilities $p_{(Q_1,Q_2),(j_1,j_2)}^{d}(k)$, $j_1 \in\{ 0,\ldots,Q_1 \}$, $i_2 \in \{ 0,\ldots,Q_2 \}$, $k \in\{0,1,\ldots \}$ must be computed. We have the following Lemma.

\begin{lem}
The transition probabilities $p_{(Q_1,Q_2),(j_1,j_2)}^{d}(k)$, $j_1 \in\{ 0,\ldots,\\ Q_1 \}$, $i_2 \in \{ 0,\ldots,Q_2 \}$, $k\in\{0,1,\ldots \}$ of the discrete time Markov chain $\{ (n_{1}^{d}(k), n_{2}^{d}(k),k\in\{0,1,\ldots \}) \}$ are given by the following formulas:\\

\begin{eqnarray}
p_{(Q_1,Q_2),(j_1,j_2)}^{d}(k)= \hspace{8.5cm} & & \nonumber
\\
\small
                              \left\{
                                \begin{array}{ll}
                                    \left(
                                     \begin{array}{c}
                                       k \\
                                       Q_1-j_1 \\
                                     \end{array}
                                   \right)
\left(\frac{\lambda_1}{s}\right)^{Q_1-j_1}\left(\frac{\lambda_2}{s}\right)^{Q_2-j_2}, & \hbox{if } j_1,j_2>0,\\
                                   & k=Q_1-j_1+Q_2-j_2 \\
                                   & \\
                                   0, & \hbox{otherwise,}
                                \end{array}
                              \right.&&
\label{3.9}
\normalsize
\end{eqnarray}

\begin{eqnarray}
p_{(Q_1,Q_2),(j_1,0)}^{d}(k)= \hspace{7.5cm} & & \nonumber
\\
\tiny
                              \left\{
                                \begin{array}{ll}
\sum_{l=Q_2}^{Q_2+Q_1-j_1}
                                    \left(
                                     \begin{array}{c}
                                       l-1 \\
                                       Q_2-1 \\
                                     \end{array}
                                   \right)
 \left(\frac{\lambda_2}{s}\right)^{Q_2}\left(\frac{\lambda_1}{s}\right)^{l-Q_2} & \\
                         \cdot    \left(
                                   \begin{array}{c}
                                     k-l \\
                                     Q_1+Q_2-j_1-l \\
                                   \end{array}
                                 \right)
 \left(\frac{s_1}{s}\right)^{Q_1+Q_2-j_1-l}\left(\frac{s-s_1}{s}\right)^{k-Q_1-Q_2+j_1}, & \\
                                   & \hbox{if }j_1>0, \\
                                   & k \geq Q_1-j_1+Q_2 \\
                                   & \\

                                   0, & \hbox{otherwise,}
                                \end{array}
                              \right.&&
\normalsize
\label{3.10}
\end{eqnarray}

\begin{eqnarray}
p_{(Q_1,Q_2),(0,j_2)}^{d}(k)= \hspace{7.5cm} & & \nonumber
\\
\small
                              \left\{
                                \begin{array}{ll}
\sum_{l=Q_1}^{Q_2+Q_1-j_2}
                                    \left(
                                     \begin{array}{c}
                                       l-1 \\
                                       Q_1-1 \\
                                     \end{array}
                                   \right)
 \left(\frac{\lambda_1}{s}\right)^{Q_1}\left(\frac{\lambda_2}{s}\right)^{l-Q_1} & \\
                           \cdot      \left(
                                   \begin{array}{c}
                                     k-l \\
                                     Q_1+Q_2-j_2-l \\
                                   \end{array}
                                 \right)
 \left(\frac{s_2}{s}\right)^{Q_1+Q_2-j_2-l} & \\
\cdot \left(\frac{s-s_2}{s}\right)^{k-Q_1-Q_2+j_2}, & \\
                                   & \hbox{if }j_2>0, \\
                                   & k \geq Q_1-j_2+Q_2 \\
                                   & \\

                                   0, & \hbox{otherwise,}
                                \end{array}
                              \right.&&
\normalsize
\label{3.11}
\end{eqnarray}

\begin{equation}
p_{(Q_1,Q_2),(0,0)}^{d}(k)=1-\sum_{(j_1,j_2)\in S\smallsetminus(0,0)}p_{(Q_1,Q_2),(j_1,j_2)}^{d}(k). \hspace{2cm}\label{3.12}
\end{equation}
\end{lem}

\begin{proof}
We want to compute the probability that the discrete time Markov chain $\{ (n_{1}^{d}(k), n_{2}^{d}(k),k\in\{0,1,\ldots \}) \}$ makes a transition from state $(Q_1,Q_2)$ to state $(j_1,j_2)$, $j_1\in\{ 0,1,\ldots,Q_1 \}$, $j_2\in\{ 0,1,\ldots,Q_2 \}$, in $k$ steps, $k\in\{ 0,1,\ldots \}$. We can do this if we determine in how many ways this transition can be made and the probability of each way.

In order to describe the transitions we will refer to a step from state $(i_1,i_2)$ to state $(i_1-1,i_2)$, $i_1\in\{ 1,\dots,Q_1 \}$, $i_2\in\{0,\ldots,Q_2\}$ as type-1 step and to a step from state $(i_1,i_2)$ to state $(i_1,i_2-1)$, $i_1\in\{ 0,\dots,Q_1 \}$, $i_2\in\{1,\ldots,Q_2\}$ as a type-2 step.

Now, if $j_1,j_2>0$ a transition from $(Q_1,Q_2)$ to $(j_1,j_2)$ can happen in exactly $k=Q_1-j_1+Q_2-j_2$ steps. More concretely, we must have $Q_1-j_1$ type-1 steps and $Q_2-j_2$ type-2 steps. That can be made in \small $\left(
                                                                                                   \begin{array}{c}
                                                                                                     k \\
                                                                                                     Q_1-j_1 \\
                                                                                                   \end{array}
                                                                                                 \right)
$\normalsize different ways and each way has probability  $\left(\frac{\lambda_1}{s}\right)^{Q_1-j_1} \cdot \left(\frac{\lambda_2}{s}\right)^{Q_2-j_2}$. Thus, we obtain \eqref{3.9}.

In order to find the probability of a transition from $(Q_1,Q_2)$ to $(j_1,0)$, $j_1>0$, in $k$ steps we must condition on the number of steps until the first passage in the set of states $\{(j_1,0),(j_1+1,0),\ldots,(Q_1,0)\}$. Assuming that the number of these steps is $l$, $l=Q_2,\ldots,Q_2+Q_1-j_1$, we have that after $l$ steps the system is at state$(Q_1+Q_2-l,0)$ and the transition from $(Q_1,Q_2)$ to $(Q_1+Q_2-l,0)$ can happen in \small $\left(
                                                                                                    \begin{array}{c}
                                                                                                      l-1 \\
                                                                                                      Q_2-1 \\
                                                                                                    \end{array}
                                                                                                  \right)
$\normalsize ways (since the last step is of type-2) and each way has probability $\left(\frac{\lambda_2}{s}\right)^{Q_2} \cdot \left(\frac{\lambda_1}{s}\right)^{l-Q_2}$. Then, the transition from $(Q_1+Q_2-l,0)$ to $(j_1,0)$ in $k-l$ steps can occur in \small$\left(
                       \begin{array}{c}
                         k-l \\
                         Q_1+Q_2-j_1-l \\
                       \end{array}
                     \right)
$\normalsize 
ways, because we must have $Q_1+Q_2-j_1-l$ type-1 steps, and each way has probability $\left(\frac{s_1}{s}\right)^{Q_1+Q_2-j_1-l}\left(\frac{s-s_1}{s}\right)^{k-Q_1-Q_2+j_1}$. So, we obtain \eqref{3.10}.

Using the same line of arguments we obtain \eqref{3.11}.

Finally, since the matrix $\mathbb{P}^k$ is stochastic, we have that\\ $\sum_{(j_1,j_2)\in S}p_{(Q_1,Q_2),(j_1,j_2)}^{d}(k)=1$. So, we have formula \eqref{3.12}.

\end{proof}

Now, the proof of Theorem $1$ is complete.

\subsection{Proof of Theorem 3 (Stationary distribution)}

For the random replenishment time case, our approach is to derive the rate matrix $\mathbb{Q}$ and consider directly the balance equations for the stationary distribution. Fix a replenishment policy $(Q_1,Q_2)$.

The rate matrix is

\begin{equation}
\mathbb{Q}^{(2)}=\left(
                   \begin{array}{cccccc}
                     \mathbb{Q}_{0}^{(2)} & \mathbb{O} & \mathbb{O} & \ldots & \mathbb{O} & \mathbb{M} \\
                     \mathbb{Q}_{1}^{(2)} & \mathbb{Q}_{2}^{(2)} & \mathbb{O} & \ldots & \mathbb{O} & \mathbb{M} \\
                     \mathbb{O} & \mathbb{Q}_{1}^{(2)} & \mathbb{Q}_{2}^{(2)} & \ldots & \mathbb{O} & \mathbb{M} \\
                     \vdots & \vdots & \vdots & \ddots & \vdots & \vdots \\
                     \mathbb{O} & \mathbb{O} & \mathbb{O} & \ldots & \mathbb{Q}_{2}^{(2)} & \mathbb{M} \\
                     \mathbb{O} & \mathbb{O} & \mathbb{O} & \ldots & \mathbb{Q}_{1}^{(2)} & \mathbb{Q}_{2}^{(2)}+\mathbb{M} \\
                   \end{array}
                 \right),\label{2.7}
\end{equation}
$\mathbb{Q}^{(2)}\in M_{Q_1 +1 \times Q_2 +1,Q_1 +1\times Q_2 +1}$, where

\begin{equation}
\mathbb{M}=\left(
                       \begin{array}{cccccc}
                         0 & 0 & 0 & \ldots & 0 & \mu \\
                         0 & 0 & 0 & \ldots & 0 & \mu \\
                         0 & 0 & 0 & \ldots & 0 & \mu \\
                         \vdots & \vdots & \vdots & \ddots & \vdots & \vdots \\
                         0 & 0 & 0 & \ldots & 0 & \mu \\
                         0 & 0 & 0 & \ldots & 0 & \mu \\
                       \end{array}
                     \right)
,\label{2.8}
\end{equation}
$\mathbb{M}\in M_{Q_2 +1,Q_2 +1}$,

\begin{eqnarray}
\mathbb{Q}_{0}^{(2)}&=&\mathbb{Q}_{0}^{(1)}-\mu\mathbb{I},\label{2.9}\\
\mathbb{Q}_{1}^{(2)}&=&\mathbb{Q}_{1}^{(1)},\label{2.9}\\
\mathbb{Q}_{2}^{(2)}&=&\mathbb{Q}_{2}^{(1)}-\mu\mathbb{I},\label{2.11}
\end{eqnarray}
where $\mathbb{Q}_{0}^{(1)}$, $\mathbb{Q}_{1}^{(1)}$ and $\mathbb{Q}_{2}^{(1)}$ are given in \eqref{2.2}-\eqref{2.4} and $\mathbb{I}\in M_{Q_2 +1,Q_2 +1}$ is the diagonal matrix with all diagonal elements equal to one.

The stationary distribution
  $(\pi_{(i_1,i_2)}:(i_1,i_2)\in S)$ is obtained as the
  unique positive normalized solution of the following system of
  balance equations:

\begin{equation}
  \pi_{(Q_1,Q_2)}\cdot(s+\mu)=\mu\sum_{i_1=0}^{Q_1}\sum_{i_2=0}^{Q_2}\pi_{(Q_1-i_1,Q_2-i_2)},\label{4.11}
\end{equation}
\begin{equation}
  \pi_{(Q_1,Q_2-i_2)}\cdot(s+\mu)=\lambda_2\cdot\pi_{(Q_1,Q_2-i_2+1)}, \ i_2=1,\ldots,Q_2-1,\label{4.12}
\end{equation}
\begin{equation}
  \pi_{(Q_1-i_1,Q_2)}\cdot(s+\mu)=\lambda_1\cdot\pi_{(Q_1-i_1+1,Q_2)}, \ i_1=1,\ldots,Q_1-1,\label{4.13}
\end{equation}
\begin{eqnarray}
  \pi_{(Q_1-i_1,Q_2-i_2)}\cdot(s+\mu)&=&\lambda_1\cdot\pi_{(Q_1-i_1+1,Q_2-i_2)}+ \lambda_2\cdot\pi_{(Q_1-i_1,Q_2-i_2+1)},\nonumber\\
  && i_1=1,\ldots,Q_1-1, \ i_2=1,\ldots,Q_2-1,\label{4.14}
\end{eqnarray}
\begin{equation}
  \pi_{(Q_1,0)}\cdot(s_1+\mu)=\lambda_2\cdot\pi_{(Q_1,1)},\label{4.15}
\end{equation}
\begin{equation}
  \pi_{(0,Q_2)}\cdot(s_2+\mu)=\lambda_1\cdot\pi_{(1,Q_2)},\label{4.16}
\end{equation}
\begin{eqnarray}
  \pi_{(Q_1-i_1,0)}\cdot(s_1+\mu)&=&s_1\cdot\pi_{(Q_1-i_1+1,0)}+ \lambda_2\cdot\pi_{(Q_1-i_1,1)},\nonumber\\
  && i_1=1,\ldots,Q_1-1,\label{4.17}
\end{eqnarray}
\begin{eqnarray}
  \pi_{(0,Q_2-i_2)}\cdot(s_2+\mu)&=&s_2\cdot\pi_{(0,Q_2-i_2+1)}+ \lambda_1\cdot\pi_{(1,Q_2-i_2)},\nonumber\\
  && i_2=1,\ldots,Q_2-1,\label{4.18}
\end{eqnarray}
\begin{equation}
  \pi_{(0,0)}\cdot\mu=s_1\cdot\pi_{(1,0)}+ s_2\cdot\pi_{(0,1)},\label{4.19}
\end{equation}
where in \eqref{4.11} we have included the pseudo-transitions from
state $(Q_1,Q_2)$ to $(Q_1,Q_2)$ with rate $\mu$, which there was no
demand. Using the normalization equation

\begin{equation}
  \sum_{i_1=0}^{Q_1}\sum_{i_2}^{Q_2}\pi_{(Q_1-i_1,Q_2-i_2)}=1,\label{4.20}
\end{equation}
\eqref{4.11} can be written as

\begin{equation}
  \pi_{(Q_1,Q_2)}\cdot(s+\mu)=\mu.\label{4.21}
\end{equation}

Note also that the continuous time Markov chain is always positive
recurrent. Using \eqref{4.11}-\eqref{4.19} after some straightforward
algebra we obtain expressions \eqref{4.1}-\eqref{4.4} for the
stationary probabilities.

\subsection{Proof of Lemma 1 ($Q_{2}^{*}(Q_1)$ is decreasing in $Q_1$)}

Since $\pi(Q_1,Q_2)$ is submodular, the following inequality yields
\begin{eqnarray}
\pi(Q_1+1,Q_2+1)-\pi(Q_1+1,Q_2)-\pi(Q_1,Q_2+1)+\pi(Q_1,Q_2)\leq 0,&& \nonumber\\ Q_1\in\{0,\ldots,\left\lfloor\frac{C}{a_1}\right\rfloor\},Q_2\in\{0,\ldots,\left\lfloor\frac{C- a_1 Q_1}{a_2}\right\rfloor\}.&&\label{5.1}
\end{eqnarray}

We will prove that
\begin{equation}
\pi(Q_1+1,Q_{2}^{'})-\pi(Q_1+1,Q_2)-\pi(Q_1,Q_{2}^{'})+\pi(Q_1,Q_2)\leq 0, \mbox{ for } Q_{2}^{'}>Q_2.\label{5.2}
\end{equation}

Adding \eqref{5.1} for $Q_2$ and \eqref{5.1} for $Q_2+1$ we obtain
\begin{eqnarray}
\pi(Q_1+1,Q_2+2)-\pi(Q_1+1,Q_2)-\pi(Q_1,Q_2+2)+\pi(Q_1,Q_2)\leq 0,&& \nonumber\\ Q_1\in\{0,\ldots,\left\lfloor\frac{C}{a_1}\right\rfloor-1\},Q_2\in\{0,\ldots,\left\lfloor\frac{C- a_1 Q_1}{a_2}\right\rfloor-2\}.&&\nonumber
\end{eqnarray}

Thus, we can prove by induction inequality \eqref{5.2}.

Now, for $Q_1=0$, we have $Q_{2}^{*}(0)\in[0,\left\lfloor\frac{C}{a_2}\right\rfloor]$ and
\begin{eqnarray}
\pi(0,Q_{2}^{*}(0))\geq \pi(0,Q_2), \text{ for all }Q_2\in[0,\left\lfloor\frac{C}{a_2}\right\rfloor].\label{5.3}
\end{eqnarray}

Also, from \eqref{5.1} for $Q_1=0$ we obtain
\begin{eqnarray}
\pi(1,Q_2+1)-\pi(1,Q_2)-\pi(0,Q_2+1)+\pi(0,Q_2)\leq 0,&&\nonumber\\
Q_2\in\{0,\ldots,\left\lfloor\frac{C}{a_2}\right\rfloor-1\}.&&\label{5.4}
\end{eqnarray}

In order to prove that $Q_{2}^{*}(1)\leq Q_{2}^{*}(0)$ it suffices to prove that\\ $\pi(1,Q_{2}^{*}(0))\geq \pi(1,Q_{2})$, $Q_2\in[Q_{2}^{*}(0),\left\lfloor\frac{C}{a_2}\right\rfloor-1]$.

From \eqref{5.2} for $Q_1=0$ and $Q_2=Q_{2}^{*}(0)$ we have that
\begin{equation}
\pi(1,Q_{2}^{'})-\pi(1,Q_{2}^{*}(0))-\pi(0,Q_{2}^{'})+\pi(0,Q_{2}^{*}(0))\leq 0, \ Q_{2}^{'}>Q_{2}^{*}(0), \nonumber
\end{equation}
which gives
\begin{equation}
\pi(1,Q_{2}^{'})-\pi(1,Q_{2}^{*}(0))\leq\pi(0,Q_{2}^{'})-\pi(0,Q_{2}^{*}(0)), \ Q_{2}^{'}>Q_{2}^{*}(0). \nonumber
\end{equation}

The second side of the previous inequality is smaller than or equal to zero because of \eqref{5.3}. So, $\pi(1,Q_{2}^{'})\leq \pi(1,Q_{2}^{*}(0))$, $Q_{2}^{'}>Q_{2}^{*}(0)$. Thus, $Q_{2}^{*}(1)\leq Q_{2}^{*}(0)$.

Using this line of arguments repeatedly, we prove that $Q_{2}^{*}(Q_1)$ is decreasing in $Q_1$.

\end{document}